\documentclass[]{article}
\usepackage{amsfonts,amsmath,amssymb}
\usepackage{amsthm}

\usepackage[english]{babel}
\allowdisplaybreaks \theoremstyle{definition}
\newtheorem{Definition}{Definition}
\newtheorem{Remark}{Remark}
\newtheorem{Example}{Example}

\theoremstyle{plain}
\newtheorem{Theorem}{Theorem}
\newtheorem{Corollary}{Corollary}
\newtheorem{Proposition}{Proposition}

\title{Metric geometry of nonregular
weighted Carnot-Carath\'eodory spaces}
\author{Svetlana Selivanova\footnote{This work was partially supported from the Integration project SB RAS -- DWO RAS, No. 56}}

\begin{document}
\large
\date{}
 \maketitle
\begin{abstract}
We investigate local and metric geometry of weighted
Carnot-Carath\'eodory spaces which are a wide generalization of
sub-Riemannian manifolds and arise in nonlinear control theory,
subelliptic equations etc. For such spaces the intrinsic
Carnot-Carath\'eodory metric might not exist, and some other new
effects take place. We describe the local algebraic structure of
such a space, endowed with a certain quasimetric (first introduced
by A. Nagel, E.M. Stein and S. Wainger), and compare local
geometries of the initial C-C space and its tangent cone at some
fixed (possibly nonregular) point. The main results of the present
paper are new even for the case of sub-Riemannian manifolds.
Moreover, they yield new proofs of such classical results as the
Local approximation theorem and the Tangent cone theorem, proved
for H\"ormander vector fields by M. Gromov, A.Bellaiche,
J.Mitchell etc.


\end{abstract}

\section{Introduction}

We investigate local and metric geometry of a general class of
Carnot-Carath\'eodory spaces (see Definition \ref{cc-space}) which
generalize classical sub-Riemannian manifolds (see e.g. \cite{bel,
herm,jean, gro,lan,mit,mon} and references therein) and naturally
arise in different areas, in particular, geometric control theory,
harmonic analysis and subelliptic equations.

As it is well-known, if $X_1,X_2,\ldots,X_m$ are smooth
``horizontal'' vector fields on a smooth connected manifold
$\mathbb M$ ($\operatorname{dim}\mathbb M=N,\ m\leq N$), a
necessary and sufficient condition for a system
\begin{equation}\label{lin}\dot{x}=\sum\limits_{i=1}^ma_iX_i(x)
\end{equation} to be locally controllable
is that $X_1,X_2,\ldots,X_m$ span, together with their
commu\-ta\-tors up to some finite order $M$, the tangent space
$T_v\mathbb M$ at any point $v\in\mathbb M$ (H\"{o}rmander's
condition \cite{hor}), i.e. define a sub-Riemannian geometry on
$\mathbb M$. The existence of a controllable ``horizontal'' path,
joining two arbitrary points $v,w\in\mathbb M$, is equivalent to
the Rashevsky-Chow connectivity theorem \cite{Chow,Rash}. This
theorem implies existence of an intrinsic Carnot-Carath\'eodory
metric $d_c(v,w)$ defined as the infimum of lengths of all
horizontal curves (with their tangent vectors belonging to the
subbundle $H\mathbb M=\operatorname{span}\{X_1,X_2,\ldots,X_m\}$)
joining $v$ and $w$. Investigation of local geometry of
sub-Riemannian manifolds is important e.g. for constructing
optimal motion planning algorithms for \eqref{lin} and studying
their complexity \cite{bel,bian,herm,jean,jean1,jean2}. In
particular, investigation of the algebraic structure of the
tangent cone (in Gromov's sense \cite{bur,gro1,gro,gro2})
$(\mathbb M, d_c^u)$ to the metric space $(\mathbb M,d_c)$ plays
here a crucial role, as well as obtaining estimates on comparison
of the metrics $d_c$ and $d_c^u$. In contrast to the Riemannian
case they are not bilipshitz-equivalent, but the following
estimate holds:

{\bf Theorem} (Local approximation theorem
\cite{bel,gro,jean,vk1}). If, for $u,v\in\mathbb M$,
$d_c(u,v)=O(\varepsilon)$ and $d_c(u,w)=O(\varepsilon)$, then
$|d_c(v,w)-d_c^u(v,w)|=O(\varepsilon^{1+\frac{1}{M}})$, where $M$
is the depth of the sub-Riemannian manifold $\mathbb M$.

If the dependence of the right-hand part of a control system is
nonlinear on the control functions  (see \cite{agmar,cor} and
references therein):
\begin{equation}\label{nonlin}\begin{cases}\dot{x}=f(x,a),\\x(0)=
x_0,\end{cases}\end{equation} where $x\in \mathbb M,a\in\mathbb
R^m$,  then a sufficient (but not necessary) controllability
condition is that
$$\operatorname{span}\bigl\{h(0):h\in\operatorname{Lie}
\frac{\partial^{|\alpha|}} {\partial
a^{\alpha}}f(0,\cdot),\alpha\in\mathbb N^M \bigr\}=T_{x_0}\mathbb
M$$ for some $M\in\mathbb N$. Letting
$$F_{\nu}=\bigl\{\frac{\partial^{\alpha}}{\partial
a^{\alpha}}f(0,\cdot):|\alpha|\leq\nu\bigr\}$$ and
$$H_k(q)=\operatorname{span}\{[f_1,[f_2,\ldots,[f_{i-1},f_i]
\ldots](q):f_j\in
F_{\nu_j}, \nu_1+\nu_2+\ldots+\nu_i\leq k, i>0\},$$ one obtains a
weighted filtration $$\{0\}\subseteq H_1\subseteq
H_2\subseteq\ldots\subseteq H_M=T\mathbb M, \text{ such that }
[H_i,H_j]\subseteq H_{i+j},$$ of the tangent bundle. The condition
of having such a filtration is obviously weaker than the
H\"{o}rmander's condition, and in this case it may happen that not
all points can be joined by a horizontal path (see Example 2),
i.e. the Rashevsky-Chow theorem fails to hold and the intrinsic
metric $d_c$ might not exist.

Other examples, where weighted Carnot-Carath\'eodory spaces
appear, stem from the theory of subelliptic equations
\cite{bram,cnsw,mm,nsw}. Besides weakening the H\"ormander's
condition, an important line of generalization of sub-Riemannian
geometry is minimizing the smoothness assumptions on the vector
fields $X_i$ generating the space (see e.g.
\cite{bram,bram1,gre,karm,karm1,
karm2,vk,morbid,morbid2,dan,sus,sus1,str,Tao,vk1,vk2}).

In this paper we consider the following notion of a weighted
Carnot-Carath\'eodory space (this definition is close to the one
of the paper \cite{cnsw}). A smooth manifold $\mathbb M$  will be
called a {\it (weighted) Carnot-Carath\'eodory space} (shortly,
C-C space) if there are $C^{2M+1}$-smooth vector fields
$X_1,X_2,\ldots,X_q$ given on an area $U\subseteq\mathbb M$ (the
number $M$ is defined below), endowed with formal weights
 $\operatorname{deg}(X_i)=d_i$, $1\leq d_1\leq
d_2\leq\ldots\leq d_q$, $d_j\in \mathbb N$, with the following
properties. It is assumed that
$\operatorname{span}\{X_I(v)\}_{|I|_h\leq M}=T_v\mathbb M$ for all
$v\in U$ and some $M\in\mathbb N$, where
$$\operatorname{deg} X_I=|I|_h=d_{i_1}+\ldots+d_{i_k}$$ is the
homogeneous degree of the commutator
$X_I=[X_{i_1},[\ldots,[X_{i_{k-1}},X_{i_k}]\ldots]$. Letting
$H_j=\operatorname{span}\{X_I\}_{|I|_h\leq j}$ we get a weighted
filtration of the tangent bundle $H_1\subseteq
H_2\subseteq\ldots\subseteq H_M=T\mathbb M,$  which meets the
property $ [H_i,H_j]\subseteq H_{i+j}. $ A point $u\in U$  is called
{\it regular} if there is a neighborhood of $u$ in which the
dimensions of all $H_k$ are constant; otherwise this point is called
{\it nonregular}.

This notion of a C-C space is suitable to describe nonlinear
control systems \eqref{nonlin}. One of the peculiarities stemming
from the presence of a formal degree structure is that different
choices of weights may lead to different distributions of regular
and nonregular points on the space (see Example 1).

Because of the mentioned difficulties, new methods for studying
local geometry of such spaces are needed. In particular, since the
metric $d_c$ might not exist, we obtain all estimates w.r.t. the
following distance function, first introduced in \cite{nsw}, which
is actually not a metric, but a quasimetric, i.e. the triangle
inequality holds only in the generalized sense, with some
constant.
$$\rho(v,w)=\inf\Bigl\{\delta>0\mid \text{ there is a curve }
 \gamma:[0,1]\to U\text{ such that }$$
$$\gamma(0)=v, \gamma(1)=w, \dot{\gamma}(t)=\sum\limits_{|I|_h\leq
M} w_I X_I(\gamma(t)), |w_I|<\delta^{|I|_h}\Bigr\}.$$

A crucial result on local geometry, which we prove in Section 5,
is the estimate on comparison of this quasimetric w.r.t. the
initial vector fields and the quasimetric $\rho^u$ (see Section
3), induced in by their nilpotentizations $\widehat{X}_i^u$ at a
point $u$, which is possibly nonregular.

{\bf Theorem} (Theorem on divergence of integral lines) If $u,v\in
U$, $\rho(u,v)=O(\varepsilon)$ and $r=O(\varepsilon)$, then we
have
$$R(u,v,r)=O(\varepsilon^{1+\frac{1}{M}}),$$
where
$$R(u,v,r)=\max\{\underset{\widehat{y}\in
B^{\rho^u}(v,r)}\sup\{\rho^u(y, \widehat{y})\},\underset{y\in
B^{\rho}(v,r)}\sup\{\rho(y, \widehat{y})\}\}.$$ Here the points
$y$ and $\widehat{y}$ are defined as follows. Let $\gamma(t)$ be
an arbitrary curve such that
$$\begin{cases}\dot{\gamma}(t)=\sum\limits_{|I|_h\leq M}b_I
\widehat{X}^u_I(\gamma(t)),\\
\gamma(0)=v, \gamma(1)=\widehat{y},
\end{cases}$$
and $$\rho^u(v,\widehat{y})\leq\underset{|I|_h\leq
M}\max\{|b_I|^{1/|I|_h}\}\leq r.$$ Define
$y=\text{exp}(\sum\limits_{|I|_h\leq M}b_I X_I)(v)$. In this way,
the supremum is taken not only over $\widehat{y}\in
B^{\rho^u}(v,r)$, but also over the infinite set of admissible
$\{b_I\}_{|I|_h\leq M}$.

This theorem allows construction of motion planning algorithms for
the system \eqref{nonlin} like it was done for \eqref{lin} in
\cite{bel,bian,herm,jean,jean1}, and to prove an analog of the
local approximation theorem, as well as to study the algebraic
structure of the tangent cone.

{\bf Theorem} (Local approximation theorem). For any points $u\in
U$ and $v,w\in U$, such that $\rho(u,v)=O(\varepsilon)$,
$\rho(u,w)=O(\varepsilon)$, we have
$$|\rho(v,w)-\rho^u(v,w)|=O(\varepsilon^{1+\frac{1}{M}}).$$

{\bf Theorem} (Tangent cone theorem). The quasimetric space
$(U,\rho^u)$ is a local tangent cone at the point $u$ to the
quasimetric space $(U,\rho)$. The tangent cone is a homogeneous
space $G/H$, where $G$ is a nilpotent graded group with a weight
structure.

This theorem, see Section 6, generalizes an analogous fact for
sub-Riemannian manifolds, known as Mitchell's cone theorem.
Namely, it is known that, at a regular point, the tangent cone to
a sub-Riemannian manifold is a nilpotent stratified group
\cite{gro,mit}, while at a nonregular point it is a homogeneous
space \cite{bel,jean}.

The notion of the tangent cone to a quasimetric space, extending
the Gromov's notion for metric spaces, was introduced and studied
recently in \cite{dan,smj}. Note that a straightforward
generalization of the Gromov-Hausdorff convergence theory would
make no sense for quasimetric spaces, since the Gromov-Haussdorff
distance between any two quasimetric spaces would be equal to
zero. However, such generalization can be done for particular
classes of compact quasimetric spaces \cite{gre}.

All of the mentioned three results are new even for the case of
``classical'' sub-Riemannian manifolds; moreover, methods of their
proofs allow to prove in a new way the classical results for
sub-Riemannian manifolds (see Section 7). In particular, in
contrast to the proof of the Local approximation theorem in
\cite{bel}, we do not need special polynomial ``privileged''
coordinates  and do not use Newton-type approximation methods.

The proofs of the main results of this paper heavily rely on
results of \cite{vk,vk1} for the case of regular C-C spaces, see
Definition \ref{cc-space-reg}, and on methods of submersion of a
C-C space into a regular one, \cite{rs,bel,cnsw,herm,jean}, as
well as on obtaining new geometric properties for the quasimetrics
$\rho$ and $\rho^u$ (Section 4).

This paper is essentially an extended version of the short notes
\cite{dan,izv}.

I am deeply grateful to Professor Sergey Vodopyanov for suggesting
me this problematic and many fruitful discussions. I am also
grateful to Doctor Maria Karmanova for a consultation on her paper
\cite{karm1}.

\section{Basic definitions, examples and known facts}

Recall that locally any vector field $X_i$ on a manifold $\mathbb
M$ can be viewed as a first-order differential operator
$X_i=\sum\limits_{j=1}^Na_{ij}(x)\frac{\partial}{\partial x_j}$
acting on a function $f\in C^{\infty}(\mathbb M)$, and its
smoothness coincides with the smoothness of the coordinate
functions $a_{ij}(x)$. A commutator of two vector fields is a
vector field defined as $[X_i,X_j]=X_iX_j-X_jX_i$.

In this paper we will use the following definition of a weighted
Carnot-Carath\'eodory space (this definition is very close to the
one formulated in \cite{cnsw}). It is easy to see that this
definition can be reformulated in such a way that it involves only
first-order  (not higher-order) commutators
 of the vector fields $X_1,X_2,\ldots, X_q$
and thus can be applied to the case of $C^1$-smooth vector fields.

\begin{Definition}\label{cc-space} Let $X_1,X_2,\ldots,X_q \in
$ be $C^{2M+1}$-smooth vector fields given on an area $U$ in a
connected $C^{\infty}$-smooth manifold $\mathbb M$  (the number
$M$ is defined below) and associated with formal weights
 $\operatorname{deg}(X_i)=d_i$, $1\leq d_1\leq
d_2\leq\ldots\leq d_q$, $d_j\in \mathbb N$. To the commutator
$X_I=[X_{i_1},[\ldots,[X_{i_{k-1}},X_{i_k}]\ldots]$ a weight equal
to its homogeneous degree is assigned:
\begin{equation}\label{deg}
\operatorname{deg} X_I=|I|_h=d_{i_1}+\ldots+d_{i_k}.
\end{equation}
It is assumed that $\operatorname{span}\{X_I(v)\}_{|I|_h\leq
M}=T_v\mathbb M$ for all $v\in U$. Letting $H_j=\operatorname{span}
\{X_I\}_{|I|_h\leq j}$  we get a weighted
filtration of the tangent bundle
\begin{equation}\label{filtr}
H_1\subseteq H_2\subseteq\ldots\subseteq H_M=T\mathbb
M,\end{equation}  which meets the property
\begin{equation}\label{filtr_commut}
[H_i,H_j]\subseteq H_{i+j}.
\end{equation}

 A manifold
$\mathbb M$ endowed with the described structure will be called a
{\it (weighted) Carnot-Carath\'eodory space} (shortly, C-C space).

The minimal number $M$ of the elements $H_i$ in the filtration
\eqref{filtr} is called  the {\it depth} of the given
Carnot-Carath\'eodory space.
\end{Definition}

Note that \eqref{deg}, \eqref{filtr_commut} relate the natural
algebraic structure, induced by commutators of the vector fields
$X_1,X_2,\ldots,X_q$, with the additional formal degree structure.

If $X_j\in C^{\infty}(U)$ and $d_1=\dots=d_q=1$ then Definition
\ref{cc-space} coincides with the classical definition of a
sub-Riemannian manifold. The subbundle $H_1$ is then called {\it
horizontal} and generates, by commutation, the whole tangent
bundle ({\it H\"ormander's condition}).

\begin{Remark}
For simplicity of notation we will carry out all computations for
the basic case when $d_1=1$, $d_q=M$. All results of this
paper hold in the framework of Definition $\ref{cc-space}$,
replacing $\frac{1}{M}$ by
$\frac{d_1}{\operatorname{max}\{d_q,M\}}$.
\end{Remark}

\begin{Definition}\label{nereg} A point $u\in U$ of a
Carnot-Carath\'eodory space is called {\it regular} if there is a
neighborhood of $u$ in which the dimensions of all $H_k$ are
constant; otherwise this point is called {\it nonregular} or {\it
singular}.\end{Definition}

\begin{Definition} Let us consider a distance function on $U$
defined as
$$\rho(v,w)=\inf\Bigl\{\delta>0\mid \text{ there is a curve }
 \gamma:[0,1]\to U\text{ such that }$$
\begin{equation}\label{rho}\gamma(0)=v, \gamma(1)=w,
\dot{\gamma}(t)=\sum\limits_{|I|_h\leq M} w_I X_I(\gamma(t)),
|w_I|<\delta^{|I|_h}\Bigr\}.\end{equation}\end{Definition}

The distance function \eqref{rho} was first introduced in
\cite{nsw} where it was proved that it is continuous and, for
``classical'' sub-Riemannian manifolds, equivalent to the
intrinsic Carnot-Carath\'eodory metric $d_c$ (Ball-Box theorem,
see also \cite{bel,jean,karm2,vk,morbid,morbid2}).

\begin{Definition}[\rm{\cite{Stein}}]\label{quasim} A {\it quasimetric
space}
$(X,d_X)$ is a topological space $X$ endowed with a quasimetric
$d_X$. A {\it quasimetric } is a mapping $d_X:X\times X\to \mathbb
R^+$ meeting the following properties

(1) $d_X(u,v)\geq 0$; $d_X(u,v)= 0$ iff $u=v$;

(2) $d_X(u,v)\leq c_Xd_X(v,u)$, where $1\leq c_X<\infty$ is a
constant independent of $u,v\in X$ (generalized symmetry
property);

(3) $d_X(u,v)\leq Q_X(d_X(u,w)+d_X(w,v))$, where $1\leq
Q_X<\infty$ is a constant independent of $u,v,w\in X$ (generalized
triangle inequality);

(4) $d_X(u,v)$ is upper-semicontinuous on the first argument.

If $c_X= Q_X=1$, then $(X,d_X)$ is a metric space.\end{Definition}

\begin{Proposition} $(U,\rho)$ is a quasimetric space.\end{Proposition}
\begin{proof}
Properties (1), (2) and (4) immediately follow from the properties
of solutions of ordinary differential equations (and we have
$\rho(v,w)= \rho(w,v)$). The generalized triangle inequality will
be proved below (Proposition \ref{treug}). 
\end{proof}

The simplest examples of (regular) weighted Carnot-Carath\'eodory
spaces are Carnot groups endowed with an additional degree
structure.

\begin{Example}[\rm{\cite{fol,fs}}]
Consider the Heisenberg group  $\mathbb H^n$: let $\mathbb
M=\mathbb R^N$, $N=2n+1$, with the coordinates
$(x_1,x_2,\ldots,x_n,y_1,y_2,\ldots,y_n,t) \in\mathbb R^N$.
Consider the vector fields
$$X_j=\partial_{x_j}-\frac{y_j}{2}\partial_{t},\
Y_j=\partial_{y_j}+\frac{x_j}{2}\partial_{t},\ \partial_t $$ with
commutator relations
$$[X_j,Y_j]=T.$$

Let us first assign to all of these vector fields the weights
naturally defined by their commutator table:
$$\text{deg}(X_j)=\text{deg}(Y_j)=1,\ \text{deg}(T)=2,$$
then for $w=\text{exp}(x_jX_j+y_jY_j+tT)(v)$ we have
$$\rho(v,w)=\max\{|x_1|,\ldots,|x_n|,|y_1|,
\ldots,|y_n|,|t|^{1/2}\}.$$ Now let $$\text{deg}(X_j)=a_j,\
\text{deg}(Y_j)=b_j,\ \text{deg}(T)=c,\text{ where } a_j+b_j=c$$
for all $j=1,\ldots,N$.  Then
$$\rho(v,w)=\max\{|x_ j|^{1/a_j},|y_j|^{1/b_j},
|t|^{1/c}\}$$ is a quasimetric not equivalent to the previous one.
In both cases all points of $\mathbb R^N$ are
regular.
\end{Example}

The next example illustrates that, for the C-C spaces from
Definition \ref{cc-space}, the Rashevsky-Chow theorem may fail to
hold, i.e. the intrinsic C-C metric might not exist.

\begin{Example}[\rm{\cite{Stein}}]\label{ex2}
Consider the Euclidean space $\mathbb R^N$ with the standard basis
$$\partial_{x_1},\partial_{x_2},\ldots,\partial_{x_N}$$ and let
$\text{deg}(\partial_{x_i})=d_j$ for
$i=k_j+1,k_j+2,\ldots,k_{j+1}$, where $k_1\leq k_2\leq\ldots\leq
k_M=N$.

Obviously, the subbundles
$$H_i=\text{span}\{\partial_{x_1},\partial_{x_2},\ldots,\partial_{x_i}\}$$
 meet the condition $[H_i,H_j]\subseteq H_{i+j}$, since
 $[H_i,H_j]=\{0\}$. At the same moment, none of the subsets of the set
  of vector fields
 $\{\partial_{x_i}\}$ meets the H\"ormander's condition, and, for
 any ``horizontal'' subbundle, there are points of $\mathbb R^N$
 which can not be joined by a horizontal curve.

In the considered example all points of $\mathbb R^N$ are regular.
If $v,w\in\mathbb R^N$ and
$w-v=(x_1,x_2,\ldots,x_N),$ then
$$\rho(v,w)=\underset{i=1,\ldots,N}\max\{|x_i|^{1/d_i}\}.
$$
\end{Example}

A further peculiarity of the considered weighted C-C spaces is that
different choices of weights $d_i$ may lead to different
combinations of regular and nonregular points.

\begin{Example}\label{reg_nereg}
Consider on the Euclidean space $\mathbb M=\mathbb R^3$ the vector
fields
 $$\{X_1=\partial_y, X_2=\partial_x+y\partial_t,X_3=\partial_x\}$$
 with the only nontrivial commutator relation $[X_1,X_2]=\partial_t$.

 Let first $\text{deg}(X_i):=1$,
$i=1,2,3$. Then $\text{deg}([X_1,X_2])=2$ and
$$H_1=\text{span}\{X_1,X_2,X_3\},\
H_2=H_1\cup\text{span}\{[X_1,X_2]\}.$$ In this case $\{y=0\}$ is a
plane consisting of nonregular points. Really, for $y\neq 0$ we
have $\text{dim}(H_1)=3$, while for $y=0$ we have
$\text{dim}(H_1)=2$.

Now assume that $\text{deg}(X_1):=a,$ $\text{deg}(X_2):=b$ and
$\text{deg}(X_3):=a+b$, where $a\leq b$. Then
$\text{deg}([X_1,X_2])=a+b$, hence
$$H_a=\text{span}\{X_1\},\ H_b=H_a\cup\text{span}\{X_2\},\
H_{a+b}=H_a\cup H_b\cup\text{span}\{X_3,[X_1,X_2]\}.$$ In this
case  all points of $\mathbb R^3$ are regular.
\end{Example}

Let us now briefly recall the approach of  the papers of S.
Vodopyanov and M. Karmanova
 \cite{karm, karm1, vk, vk1, vk2}, devoted to regular C-C spaces
 (they are particular cases of weighted C-C spaces from Definition
  \ref{cc-space}) in minimal smoothness assumptions, and some main
  results of those papers, on which the proofs of the main results
  of the present paper heavily rely.

\begin{Definition}[\rm \cite{vk, vk1, vk2, gre,karm,karm1}, cf.
\cite{bel,gro,lan,mm,nsw, rs} etc.]\label{cc-space-reg}
 Let $\mathbb M$ be a connected $C^{\infty}$-smooth Riemannian
 manifold of dimension $N$.  The manifold $\mathbb M$
is called a  {\it regular Carnot-Carath\'eodory space},
if there is a filtration of its tangent bundle $T\mathbb M$
\begin{equation}\label{filtr1}
H\mathbb M=H_1\subsetneq \ldots\subsetneq H_i\subsetneq\ldots
H_M=T\mathbb M,\end{equation} such that in some area
$U\subseteq\mathbb M$ there are $C^{p}$-smooth vector fields
$X_1,\ldots,X_N$, where $p>1$, meeting the following conditions.

\noindent For all $u\in U$ we have

(i) $H_i(u)=\text{span}\{X_1(u),\ldots,X_{\text{dim} H_i}(u)\}$
is a subbundle of $T_u\mathbb M$ of dimension
 $\text{dim}H_i,\
i=1,\ldots,M$;

(ii) The following decomposition holds
 \begin{equation}\label{commut}[X_i,X_j](v)=
 \sum\limits_{k:\deg X_k\leq \deg X_i+\deg
X_j}c_{ijk} (v) X_k(v),\end{equation} where
 $\deg X_k=\min\{m|X_k\in H_m\}$ is the {\it degree}
of the vector field $X_k$.

The number $M$ is again called the {\it depth} of the C-C space
$\mathbb M$.
\end{Definition}

The condition (i) is equivalent to \eqref{filtr_commut} in Definition
\ref{cc-space}.

\begin{Remark}
In the present paper it suffices to have, for regular C-C spaces,
smoothness $p=M+1$, but most of the results of this section are
true for $C^{1,\alpha}$-smooth vector fields $X_1,\ldots,X_N$,
where $\alpha>0$ is the H\"older constant of the first-order
derivatives. In this case, the expression $\frac{1}{M}$ in the
estimates below is replaced by $\frac{\alpha}{M}$.
\end{Remark}

Consider on $U\subseteq\mathbb M$ {\it canonical first-order coordinates}
  defined in a neighborhood of a point
$g\in\mathbb M$ as
\begin{equation}\label{exp}
\theta_g(v_1,\ldots,v_N)=\exp\left(\sum\limits_{i=1}^Nv_iX_i\right)(g).
\end{equation}
From theorems on continuous dependence of the solutions of ODE on
the initial data (see e.g. \cite{pon}) it follows that $\theta_g$ is
a $C^1$-diffeomorphism of the Euclidean ball $B_E(0,r)\subset
\mathbb R^N$ to $\mathbb M$, where $0\leq r<r_g$ for a sufficiently
smooth $r_g>0$. Let $U_g=\theta_g(B_E(0,r_g))$. The tuple
$(v_1,v_2,\ldots, v_N)=\theta_g^{-1}(v)\in B_E(0,r_g)$ is called the
first-order coordinates of the point $v\in U_g$. Further we assume
that $U\subseteq\bigcup_{g\in U}U_g$.

In the regular case, the tuple  $(v_1,v_2,\ldots, v_N)$ is
uniquely defined, thus the quasimetric \eqref{rho}, denoted
in the above-mentioned papers as $d_{\infty}$, is defined for points
$w,v\in {\cal U}$, such that
$v=\exp\left(\sum\limits_{i=1}^Nv_iX_i\right)(w)$, as
$$d_{\infty}(w,v)=\rho(w,v)=\max\limits_i\{|v_i|^{\frac{1}{\deg X_i}}\}. $$

 The generalized triangle inequality for $d_{\infty}$ is proved in
 \cite{karm,vk,vk2} in minimal smoothness assumptions and in \cite{nsw}
 for sufficiently smooth vector fields (in the general case, not just
 near regular points).


The balls w. r. t. the quasimetric $d_{\infty}$ will be denoted as
$\operatorname{Box}(u,r)=\{v\in U\ |\ d_{\infty}(v,u)<r\}.$

The vector fields $\widehat{X}_i^u$, obtained from the commutator
table \eqref{commut} replacing the inequality by equality, i. e.
$$[\widehat{X}^u_i,\widehat{X}^u_j](v)=
 \sum\limits_{k:\deg X_k= \deg X_i+\deg
X_j}c_{ijk} (u) \widehat{X}_k(v),$$ define a graded nilpotent Lie
algebra $V=V_1\oplus\ldots\oplus V_M$, where $[V_1,V_i]\subseteq
V_{i+1},\ i=1,\ldots M-1$, due to the following result.

\begin{Theorem}[\rm \cite{vk}]\label{al}
For a fixed point $u\in  U$ consider a family of coefficients
$$\{\bar{c}_{ij}^k\}=\{c_{ijk}(u)\}_{\operatorname{deg}X_k=
\operatorname{deg}X_i+\operatorname{deg}X_j}.$$

Then these constants $\{\bar{c}_{ij}^k\}$ meet the Jacobi identity
and hence define a Lie algebra. This Lie algebra is graded and nilpotent
and it can be defined by canonical
$C^{\infty}$-smooth vector fields
$\{(\widehat{X}_i^g)^{\prime}\}_{i=1}^N$ on $\mathbb R^N$, such that
$(\widehat{X}_i^g)^{\prime}(0)=e_i$.

By means of the exponential mapping \eqref{exp} the obtained
vector fields can be pushed into the initial space:
$D\theta_g\langle (\widehat{X}_i^g)^{\prime}
\rangle=\widehat{X}_i^g\in C^{\alpha}({\cal U})$ in such way that
\begin{equation}\label{equal_g}\widehat{X}_i^g(g)=X_i(g),\quad i=1,
\ldots,N.\end{equation}
\end{Theorem}

\begin{Definition}\label{loc_carno} Denote the local Lie group,
corresponding to the Lie algebra $V$ generated by
$\{\widehat{X}_i^g\}_{i=1}^N$, as ${\cal G}^g=(U,\star)$ at the
point $g$. The group operation $\star$ is defined as follows: if
$$x=\exp\Bigl(\sum\limits_{i=1}^Nx_i\widehat{X}^g_i\Bigr)(g),\
y=\exp\Bigl(\sum\limits_{i=1}^Ny_i\widehat{X}^g_i\Bigr)(g),$$ then
$$x\star
y=\exp\Bigl(\sum\limits_{i=1}^Ny_i\widehat{X}^g_i\Bigr)\circ
\exp\Bigl(\sum\limits_{i=1}^Nx_i\widehat{X}^g_i\Bigr)(g) =
\exp\Bigl(\sum\limits_{i=1}^Nz_i\widehat{X}^g_i\Bigr)(g),$$ where
$z_i$ are calculated by means of the Campbell-Hausdorff
formula.\end{Definition}

Note that, if the H\"ormander's condition holds, ${\cal G}^g$ is a
local Carnot group, i. e. $V=V_1\oplus\ldots\oplus V_M$, where
$[V_1,V_i]= V_{i+1},\ i=1,\ldots M-1$.

The quasimetric on ${\cal G}^g$ is defined in a similar way as
$d_{\infty}$: for $u,v\in {\cal G}^g$ such that
$v=\exp\left(\sum\limits_{i=1}^Nv_i\widehat{X}^g_i\right)(u)$ let
$$d_{\infty}^g(u,v)=\rho^g(u,v)=
\max\limits_i\{|v_i|^{\frac{1}{\deg X_i}}\}.$$

In the present paper we will use the following results.

\begin{Theorem}[\rm \cite{vk}]\label{decomp}
 For all $x\in U$, such that $|x_j|\leq
 \varepsilon^{|I_j|}$, the following decompositions hold:
  \begin{equation}\label{dec}
 X_j(x)=\sum\limits_{i=1}^{\tilde{N}}a_{j,k}(x)\widehat{X}_k(x),
 \end{equation}
 where
 $$ a_{j,k}=
 \begin{cases}
\delta_{j,k}+O(\varepsilon),\quad&\operatorname{deg}(X_j)=\operatorname
{deg}(X_k),\\
O(\varepsilon),\quad&\operatorname{deg}(X_j)<\operatorname{deg}(X_k),\\
o(\varepsilon^{|I_k|-|I_j|}),\quad&\operatorname{deg}(X_k)>
\operatorname{deg}(X_j).
 \end{cases}
 $$
 \end{Theorem}

Note that Theorem \ref{decomp} implies  Gromov's nilpotentization
theorem, which is proved in \cite{gro,bel,rs,vk} for smooth vector
fields, in  \cite{vk} for $C^1$ vector fields and depth $M=2$, in
\cite{gre} for $C^2$ vector fields, in \cite{karm1} for
$C^{1,\alpha}$ vector fields, where $\alpha>0$.

\begin{Theorem}[\rm Theorem on divergence of integral lines \cite{vk1}]
\label{int_lines_perem} Let $u,v\in U$,
$d_{\infty}(u,v)=C\varepsilon$, $C<\infty$. Consider the curves
$\gamma(t)$, $\widehat{\gamma}(t)$ in
$\operatorname{Box}(v,K\varepsilon)$, satisfying the equations
$$\begin{cases}\dot{\gamma}(t)=\sum\limits_{i=1}^Nb_i(t)X_i(\gamma(t)),\\
\gamma(0)=v,\end{cases}\quad\text{and}\quad
\begin{cases}\dot{\widehat{\gamma}}(t)=
\sum\limits_{i=1}^Nb_i(t)\widehat{X}^u_i(\widehat{\gamma}(t)),\\
\widehat{\gamma}(0)=v,\end{cases}$$ where
$$\int\limits_{0}^1|b_i(t)|dt<S\varepsilon^{\operatorname{deg}X_i},
S<\infty.$$
 Then
$$\max\{d_{\infty}(w,\widehat{w}),d_{\infty}^u(w,\widehat{w})\}=
O(\varepsilon^{1+\frac{1}{M}})$$ uniformly on $U$.
\end{Theorem}

Note that in \cite{karm,karm1} an analog of this   result (with
constant coefficients $b_i$) is proved without using the
Campbell-Hausdorff formula and Gromov's nilpotentization theorem,
for the case of $C^{1,\alpha}$-smooth vector fields, by means of
estimates obtained in \cite{karm,vk}.



Theorem \ref{int_lines_perem} and its analogs have many important
corollaries, in particular, each of them allows to prove the local
approximation theorem, in the smoothness assumptions considered in
each case, and also the Ball-Box theorem in the framework of the
following definition.

\begin{Definition} [\rm\cite{vk}]\label{cc-manifold}
If in Definition \ref{cc-space-reg} the following assumption (3)
holds, then $\mathbb M$ is called a {\it Carnot manifold}.

(3) The factor-mapping $[\,\cdot ,\cdot\,]:H_1\times
H_j/H_{j-1}\to H_{j+1}/H_{j}$, induced by the Lie bracket, is an
epimorphism for all $1\leq j<M$ (here it is assumed that
$H_0=\{0\}$).

In this case, the subbundle $HM=H_1$ is called {\it horizontal}.
\end{Definition}

By means of Theorem \ref{int_lines_perem}, an analog of the
Rashevsky-Chow theorem is proved in \cite{karm,vk} for Carnot
manifolds defined by $C^{1,\alpha}$-smooth vector fields. Thus it
is possible to define the intrinsic C-C
metric\begin{equation}\label{dc}d_c(u,v)=\inf\limits_
{\substack{\dot{\gamma}\in H\mathbb M, \\ \gamma(0)=u,\
\gamma(1)=v}}\{L(\gamma)\}.\end{equation}



The following assertion is formulated and proved in \cite{vk1}, in
the proof of the local approximation theorem.

\begin{Theorem}\label{int_lines_dc}
Consider the curves $\gamma$ and $\widehat{\gamma}$, satisfying
the equations
$$\begin{cases}\dot{\gamma}(t)=\sum\limits_{i=1}^ma_i(t)
\xi_i(\gamma(t)),\\
\gamma(0)=\tilde{v},
\end{cases}
\quad\text{and}\quad \begin{cases}\dot{\widehat{\gamma}}\
(t)=\sum\limits_{i=1}^ma_i (t)\widehat{\xi
}_i(\widehat{\gamma}(t)),\\
\widehat{\gamma}(0)=v.
\end{cases}$$
Denote $\gamma(1)=w$, $\widehat{\gamma}(1)=\widehat{w}$. If we
have  $d_c(u,v)=O(\varepsilon)$ and $d_c(v,w)=O(\varepsilon)$,
then
\begin{equation}\label{int_lines_est}\max\{d_c(w,\widehat{w}),d_c^{u}(w,
\widehat{w})\}=O(\varepsilon^{1+\frac{1}{M}}).\end{equation}
\end{Theorem}

\begin{Theorem}[\rm Local approximation theorem \cite{vk1}]
\label{lat_dc}  Uniformly on $u\in U$, $v,w\in
B^{d_c}(u,\varepsilon)$ the following estimate holds
$$\left|d_c(v,w)-d_c^u(v,w)\right|=
O(\varepsilon^{1+\frac{1}{M}}).$$
\end{Theorem}

\section{Choice of basis, nilpotent approximation and
a homo\-ge\-neous quasimetric}

\begin{Definition}\label{min_bas}
Among the vector fields $\{X_I\}_{|I|_h\leq M}$ we choose a basis
\begin{equation}\label{bas}
\{Y_1, Y_2,\ldots, Y_N\}
\end{equation}
  as follows:

(i) the vector fields $Y_1,Y_2,\ldots, Y_N$ are linearly
independent at the point $u$ (hence, in some neighborhood of $u$);

(ii) the sum of their weights $\sum\limits_{i=1}^N \text{deg} Y_i$
is minimal;

(iii) the sum of orders $\sum\limits_{j=1}^N |I_j|$ of the
commutators $X_{I_j}$, corresponding to $Y_j$, is
minimal.

We say that the basis meeting conditions (i), (ii), (iii) is {\it
associated with the filtration} \eqref{filtr} at the point $u$.
\end{Definition}

 Denote the dimension of the $k$-th element $H_k$ of filtration \eqref{filtr}  at the point $u$ as $n_k=\operatorname{dim}H_k(u)$.
Then items (i), (ii) of Definition \ref{min_bas} are equivalent to
the fact that the vectors $\{Y_1(u),\ldots,Y_{n_k}(u)\}$ form
bases of $H_k(u)$ for all $k=1,\ldots,M$.

\begin{Remark}
Bases satisfying (i), (iii) were considered for ``classical''
sub-Riemannian geometry in \cite{bel,jean,mon} and other papers
(``normal'' or ``mimimal''  frame), when (ii) and (iii) coincide. In
our case the necessity of considering both (ii) and (iii) can be
seen from the Example \ref{reg_nereg}: having only (i), (ii) we can
choose both the basis $\{X_1,X_2,X_3\}$ and $\{X_1,X_2,[X_1,X_2]\}$;
these bases define a different algebraic structure. Adding both
conditions excludes such examples.
\end{Remark}
\begin{Proposition}\label{treug_prop}
For any vector field $X\in H_s$ we have
\begin{equation}\label{treugX}
X(v)=\sum\limits_{i=1}^N\xi_i(v)Y_i(v), \text{ where }\xi_i(u)=0
\text{ for }
 \operatorname{deg}Y_i>s.\end{equation}\end{Proposition}
\begin{proof}
Really, by choice of the basis \eqref{bas} the vectors
$Y_1(u),\ldots, Y_{n_s}(u)$ constitute a basis of $H_s(u)$, hence
$\xi_i(u)=0$ for $i>n_s$. Consequently,  $\xi_i(u)=0$ for
 $\operatorname{deg}Y_i>s$.
\end{proof}

 \begin{Proposition}\label{treug_prop_reg} At a fixed point $u\in U$
 the following identity holds:
\begin{equation}\label{treug}[Y_i,Y_j](u)=\sum\limits_
{\operatorname{deg}Y_k\leq
\operatorname{deg}Y_i+\operatorname{deg}Y_j}
c_{ijk}(u)Y_k(u).\end{equation} If the point $u$ is regular, this
identity holds not just in $u$, but in some neighborhood of $u$.
\end{Proposition}
\begin{proof}
The identity \eqref{treug} follows from the fact that
$[H_m,H_l]\subseteq H_{m+l}$.

In some neighborhood of a regular point we can choose the same
basis, satisfying (i), (ii), (iii), for all points, by definition of
regularity.
\end{proof}


\begin{Definition} Consider {\it second-kind canonical coordinates}
$\Phi^u:\mathbb R^N\to U$ on $U$ defined as
\begin{equation}\label{loc_coord}\Phi^u(x_1,\ldots,x_N)
 =\operatorname{exp}(x_1Y_1)\circ
 \operatorname{exp}(x_2Y_2)\circ\ldots\circ\operatorname{exp}(x_NY_N)
 (u)\end{equation}
\end{Definition}
Due to the smoothness assumptions of the Definition \ref{cc-space}
and theorems on continuous dependence of solutions of ordinary
differential equations on parameters \cite{pon}, the mapping $\Phi^u$
is a
$C^{M+1}$-diffeomorphism onto some neighborhood of zero
$V\subseteq\mathbb R^N$.

We will construct nilpotent approximations in these coordinates
\eqref{loc_coord} in the same way as it was done in
\cite{bel,herm}. Dilations are defined like in \cite{bel,fs,herm}:
on $\mathbb R^N$ let $\delta_{\varepsilon} (x_1,x_2,\ldots,
x_N)=(x_1\varepsilon^{\operatorname{deg}Y_1},x_2\varepsilon^
{\operatorname{deg}Y_2},
\ldots,x_N\varepsilon^{\operatorname{deg}Y_N}).$ The function
$f:\mathbb R^N\to\mathbb R$ is {\it homogeneous of order} $l$, if
$f(\delta_{\varepsilon}x)=\varepsilon^lf(x).$ 

\begin{Definition}\label{odnor_vf}
A vector field $X$ on $\mathbb R^N$ is homogeneous of order $s$,
if $\delta_{\varepsilon}^*X=\varepsilon^sX,$ where the action of
dilations on a vector field is defined as
$\delta_{\varepsilon}^*X(f\circ\delta_{\varepsilon})=(Xf)
\circ\delta_{\varepsilon}.$
\end{Definition}




The proofs of the next Proposition \ref{vf_decomp} and Corollary
\ref{alg_prop} follow the scheme of \cite{herm} for $C^{\infty}$
vector fields meeting the H\"ormander's condition. We recall
briefly main steps of these proofs.
\begin{Proposition}\label{vf_decomp}
In coordinates $\Phi^u$ for the   $C^{M+1}$-smooth vector field $X_I$
the following decomposition holds:
$$X_I^{\prime}(x):=(\Phi^u)^{-1}_*X_I(\Phi^u(x))=\sum\limits_{j=1}^Na_j(x)
\frac{\partial}{\partial x_j}=$$ $$=
\sum\limits_{j=1}^N\left(\sum\limits_{\substack{|\alpha|_h\geq
\operatorname{deg}Y_j- \operatorname{deg}X_I, \\ |\alpha|\leq
M}}f_{(j,\alpha)} x^{\alpha}+o(||x||^M)\right)
\frac{\partial}{\partial x_j} \text{ for }||x||\to 0,$$ where
$\alpha=(\alpha_1,\ldots,\alpha_N)$, $|\alpha|_h=
\sum\limits_{i=1}^N\alpha_ i\operatorname{deg}Y_i$, $|\alpha|=
\sum\limits_{i=1}^N\alpha_ i$, $f_{(j,\alpha)} \in\mathbb R$,
$||x||$ is the Euclidean norm in $\mathbb R^N$.
\end{Proposition}
\begin{proof} Applying to both parts of the obvious equality
$$\Phi^u_*X_I^{\prime}(x)=X_I(\Phi_u(x)),\ x\in V\subseteq \
\mathbb R^N$$ the mapping $\operatorname{exp}(-x_NY_N)_*
\ldots\operatorname{exp}(-x_1Y_1)_*$ and carrying out all the
differentiations \cite{herm} in the obtained equality
$$\sum\limits_{j=1}^Na_j(x)\operatorname{exp}(-x_NY_N)_*
\ldots\operatorname{exp}(-x_1Y_{1})_*\Phi^u_*\left(\frac{\partial}
{\partial x_j}\right)=$$
\begin{equation}\label{tozhd}=\operatorname{exp}(-x_NY_N)_*
\ldots\operatorname{exp}(-x_1Y_1)_*X_I(\Phi^u(x)),\end{equation}
we get the identity
$$\sum\limits_{|\nu|=0}^M\frac{(-x_1)^{\nu_1}}{\nu_1!}\frac{(-x_2)^
{\nu_2}}{\nu_2!}
\ldots\frac{(-x_N)^{\nu_N}}{\nu_N!}(\operatorname{ad}^{\nu_N}Y_N\ldots
\operatorname{ad}^{\nu_2}Y_2
\operatorname{ad}^{\nu_1}Y_1,X_I)(u)+o(||x||^M)=$$
\begin{equation}\label{tozhd1}
=\sum\limits_{|\nu|=0}^Ma_j(x)\frac{(-x_{j+1})^{\nu_{j+1}}}{\nu_{j+1}!}
\ldots\frac{(-x_N)^{\nu_N}}{\nu_N!}(\operatorname{ad}^{\nu_N}Y_N\ldots
\operatorname{ad}^{\nu_{j+1} }Y_{j+1},Y_j)(u)+o(||x||^M),
\end{equation}
where $$(\operatorname{ad}Z,Y)=[Z,Y];\
(\operatorname{ad}^{\nu+1}Z,Y)=[Z,(\operatorname{ad}^{\nu}Z,Y)];\
[\operatorname{ad}^0Z,Y]=Y.$$ According to Proposition
\ref{treug_prop}, the following decomposition holds:
$$(\operatorname{ad}^{\nu_N}Y_N\ldots\operatorname{ad}^{\nu_2}Y_2
\operatorname{ad}^{\nu_1}Y_1,X_I)(u)=\sum\limits_{k=1}^N\beta_{\nu}^k
Y_k(u),$$ where
\begin{equation}\label{beta}
\beta_{\nu}^k=0 \text{ for }
|\nu|_h=\sum\limits_{j=1}^N\nu_j\operatorname{deg}Y_j<\operatorname{deg}
Y_k-\operatorname{deg}X_I.
\end{equation}

Denoting
$$b_k(x)=\sum\limits_{|\nu|=0}^M\beta_{\nu}^k\frac{(-x_1)^{\nu_1}}
{\nu_1!}\frac{(-x_2)^{\nu_2}}{\nu_2!}
\ldots\frac{(-x_N)^{\nu_N}}{\nu_N!}$$
and
$$c_{jk}(x)=\frac{(-x_{j+1})^{\nu_{j+1}}}{\nu_{j+1}!}
\ldots\frac{(-x_N)^{\nu_N}}{\nu_N!}\gamma_{\nu,j}^k,\quad\text{where }(\operatorname{ad}^{\nu_N}Y_N\ldots\operatorname{ad}^{\nu_{j+1}
}Y_{j+1},Y_j)(u)=\sum\limits_{k=1}^N\gamma_{\nu,j}^kY_k(u),$$
we derive
\begin{equation}\label{lev}\sum\limits_{j=1}^Na_j(x)\left[Y_j(u)+
\sum\limits_{k=1}^Nc_{jk}(x)Y_k(u)\right]=\sum\limits_{k=1}^Nb_k(x)Y_k(u)+o(||x||^M)
 \text{ for }||x||\to 0,
\end{equation}
Here $b_k(x)$ is a polynomial function beginning from terms $x_1^{\nu_1}\ldots x_N^{\nu_N}$ of order
$|\nu|_h\geq\operatorname{deg}Y_k-\operatorname{deg}X_I$, while $||(c_{jk}(x))||<1$ in some neighborhood of zero. Denoting
$$a(x)=(a_1(x),a_2(x),\ldots,a_N(x)),$$
$$b(x)=(b_1(x),b_2(x),\ldots,b_N(x)),\ C(x)=(c_{jk}(x))_{j,k=1}^N,$$
we finally obtain
$$a(x)=(I+C(x))^{-1}(b(x)+o(||x||^M))=b(x)-C(x)b(x)+o(||x||^M)
\text{ for }||x||\to 0,$$ from where, according to the properties
of $b_k(x)$, the proposition follows.
\end{proof}

Since $\operatorname{deg}X_I=|I|_h$ and the vector field $\frac{\partial}
{\partial x_j}$ is homogeneous of order $-\operatorname{deg}Y_j$, we have
\begin{Corollary}
The vector field $X_{I}^{\prime}\in C^{M}$, $|I|_h\leq M$, can be written as
$$X_{I}^{\prime}(x)=(X_{I}^{\prime})^{(-|I|_h)}(x)+(X_{|I|_h}^{\prime})^
{(-|I|_h+1)}(x)+
\ldots+(X_{I}^{\prime})^{(-|I|_h+M)}(x)+o(||x||^M) \text{ for
}||x||\to 0,$$ where the $C^{\infty}$-smooth vector field
$(X_{I}^{\prime})^{(-j)}$ is homogeneous of order $-j$.
\end{Corollary}

\begin{Corollary}\label{alg_prop}
The $C^{M+1}$-smooth vector fields
$\{\widehat{X}^u_{I}\}_{|I|_h\leq M}$ on $\mathbb M$,
where $\widehat{X}^u_{I}=
\Phi^u_*\langle(X_{I}^{\prime})^{(-|I|_h)}\rangle$,
 constitute a nilpotent Lie algebra
\begin{equation}\label{lie}L=\operatorname{Lie}
\{\widehat{X}^u_1,\ldots,\widehat{X}^u_q\}\end{equation}
and we have $$H_l(u)=\widehat{H}_l(u), \quad\text{ where }
\widehat{H}_l=\operatorname{span}\{\widehat{X}^u_I\}_{|I|_h\leq l}.$$
The vector fields $\{\widehat{Y}^u_1,\widehat{Y}^u_2,\ldots,\widehat{Y}^u_N\},$
chosen from the commutators $\widehat{X}^u_{I}$ in the same way as the basis
\eqref{bas} from the commutators $X_{I}$, form a basis, associated with the filtration \eqref{filtr} on some neighborhood $\widehat{U}$ of the point $u$.
\end{Corollary}
\begin{proof}
The smoothness assertion follows from the fact that $\Phi^u$
is a $C^{M+1}$-diffeomorphism and that
$\Phi^u_*[X,Y]=[\Phi^u_*X,\Phi^u_*Y].$ The Lie algebra is nilpotent since for $|I|_h>M$ we have
 $(X_I^{\prime})^{(-|I|_h)}=0$.

To prove the second part of the corollary it is sufficient to note
that $(\widehat{Y}_i^u)^{\prime}(0)=\frac{\partial}{\partial x_i}$
due to differentiation rules and homogeneity of the vector fields
$(\widehat{Y}_i^u)^{\prime}$. Thus the vector fields
$\widehat{Y}_i^u$ are linearly independent at the point $u$ and
hence in some its neighborhood. Moreover, if
$$X_{I}(v)=\sum\limits_{i=1}^N\xi_i(v)Y_i(v),\
\widehat{X}_I(v)=\sum\limits_{i=1}^N\eta_i(v)Y_i(v),$$ then
$\xi_i(u)=\eta_i(u)$ for $n_{|I|_h-1}+1\leq i\leq N$. Indeed, in
coordinates \eqref{loc_coord} we have
$$X_I^{\prime}(0)=(\Phi^u_*)^{-1}X_I(u)=\sum\limits_{i=1}^N\xi_i(u)
\frac{\partial} {\partial x_i};\
\widehat{X}_I^{\prime}(0)=(\Phi^u_*)^{-1}\widehat{X}_I(u)=
\sum\limits_{i=1}^N\eta_i(u)\frac{\partial} {\partial x_i};$$
$$X_I^{\prime}(x)=\widehat{X}_I^{\prime}(x)+Z(x),$$ where the vector field
$Z(x)=\sum\limits_{i=1}^Nz_i(x)\frac{\partial} {\partial x_i}$
consists of summands having order of homogeneity bigger than
$-|I|_h$, hence $z_i(0)=0$ for $n_{|I|_h-1}+1\leq i\leq N$.
\end{proof}

W.l.o.g. assume that $U=\widehat{U}$.

\begin{Definition}
The vector fields $\{\widehat{X}^u_{I}\}_{|I|_h\leq M}$ are called
{\it nilpotent approximations} of the vector fields
$\{X_I\}_{|I|_h\leq M}$.
\end{Definition}

\begin{Definition}
Define a dilation group, associated with the basis \eqref{bas},
$\Delta^v_{\varepsilon}=\Phi^v\delta^v_{\varepsilon}(\Phi^v)^{-1}$
on $U$: if
$$w=\operatorname{exp}(w_1Y_1)
\circ\operatorname{exp}(w_2Y_2)\circ\ldots\circ\operatorname{exp}(w_NY_N)
(v),$$ then
\begin{equation}\label{dil}\Delta_{\varepsilon}^v w=\operatorname{exp}(
w_1\varepsilon^
{\operatorname{deg}Y_1}Y_1)\circ\operatorname{exp}(w_2\varepsilon^
{\operatorname{deg}Y_2}Y_2)
\circ\ldots\circ\operatorname{exp}(w_N\varepsilon^{\operatorname{deg}
Y_N}Y_N)(v).\end{equation}
\end{Definition}

From Proposition \ref{vf_decomp} it follows immediately

\begin{Corollary}\label{gromov}
On $U$ the following convergence takes place:
$$(\Delta^u_{\varepsilon^{-1}})_*\varepsilon^{|I|_h}X_{I}
(\Delta^u_{\varepsilon}(v))\to\widehat{X}_{I}^u(v)
\text{ for } \varepsilon\to 0, |I|_h\leq M.$$
\end{Corollary}
\begin{proof}
Really, in coordinates \eqref{loc_coord} we have
$$(\delta_{\varepsilon^{-1}})_*\varepsilon^{|I|_h}X_{I}^{\prime}
(\delta_{\varepsilon}(x))$$$$=\delta_{\varepsilon}^*
\varepsilon^{|I|_h}\sum\limits_{j=1}^N\left(\sum
\limits_{\operatorname{deg}Y_j- |I|_h\leq|\alpha|_h\leq
M}f_{(j,\alpha)} (\varepsilon x)^{\alpha}+o(||\varepsilon x||^M)
\right)\frac{\partial}{\partial x_j}
$$
$$=
\varepsilon^{|I|_h}\sum\limits_{j=1}^N\varepsilon^{\operatorname{deg}Y_j}
\left(\sum\limits_{\operatorname{deg}Y_j- |I|\leq|\alpha|_h\leq
M}f_{(j,\alpha)} \varepsilon^{|\alpha|_h}
x^{\alpha}+o(||\varepsilon x||^M) \right)\frac{\partial}{\partial
x_j}\to$$ $$\to \sum\limits_{j=1}^N
\left(\sum\limits_{|\alpha|_h=\operatorname{deg}Y_j-
|I|_h}f_{(j,\alpha)} x^{\alpha} \right)\frac{\partial}{\partial
x_j}=(X_I^{\prime})^{(-|I|_h)}$$ for $\varepsilon\to 0$.
\end{proof}

Introduce a distance function on $U$, generated by nilpotent
approximations, in a similar way as in \eqref{rho}:
$$\rho^u(v,w)=\inf\{\delta>0\mid \text{ there is a curve}
\gamma:[0,1]\to U, \text{ such that }$$
\begin{equation}\label{rho_u}\gamma(0)=v, \gamma(1)=w,
\dot{\gamma}(t)=\sum\limits_{|I|_h\leq M} w_I
\widehat{X}^u_I(\gamma(t)), |w_I|<\delta^{|I|_h}\}.\end{equation}

Actually, $\rho^u$ is again a quasimetric; the generalized triangle
inequality will be proved in the next subsection.

\begin{Proposition}
The quasimetric $\rho^u$ meets the conical property
\begin{equation}\label{cone_prop}\rho^u(\Delta_{\varepsilon}^uv,\Delta_
{\varepsilon}^uw)= \varepsilon \rho^u(v,w).
\end{equation}
\end{Proposition}
\begin{proof}
By definition, $\rho^u(v,w)$ is the infimum of
$\underset{|I|_h\leq M}\max\{|a_I|^{1/|I|_h}\}$ over all curves
$\gamma$ such that
$$\begin{cases}
\dot{\gamma}(t)=\sum\limits_{|I|_h\leq M} a_I
\widehat{X}^u_I(\gamma(t)),\\ \gamma(0)=v, \gamma(1)=w.
\end{cases}$$
 Consider the curve
\begin{equation}\label{gamma_eps}
\gamma_{\varepsilon}(t)=\Delta_{\varepsilon}^u\gamma(t).\end{equation}
Due to homogeneity of the vector fields $\widehat{X}^u_I$ we have
$$\begin{cases}
\dot{\gamma_{\varepsilon}}(t)=\sum\limits_{|I|_h\leq M}
a_I\varepsilon^{|I|_h}
\widehat{X}^u_I(\gamma_{\varepsilon}(t)),\\
\gamma_{\varepsilon}(0)=\Delta_{\varepsilon}^uv,
\gamma_{\varepsilon}(1)=\Delta_{\varepsilon}^uw.
\end{cases}$$ Note that all curves connecting the points
$\Delta_{\varepsilon}^uv$ � $\Delta_ {\varepsilon}^uw$ have the form
\eqref{gamma_eps}: really, let $\kappa(t)$ be an arbitrary curve connecting the points $\Delta_{\varepsilon}^uv$ and
$\Delta_ {\varepsilon}^uw$, then the curve $\gamma(t)=\Delta_
{\varepsilon^{-1}}^u\kappa(t)$ connects the points $v$ and $w$. Hence
$ \rho^u(\Delta_{\varepsilon}^uv,\Delta_ {\varepsilon}^uw)$ is the infimum of $\varepsilon\underset{|I|_h\leq
M}\max\{|a_I|^{1/|I|_h}\}$ over $\gamma$, from where the proposition follows.
\end{proof}

\section{The lifting construction and further properties of
 the quasimetrics $\rho$ and $\rho^u$}

In this section we first recall the lifting construction proposed
in by L.~Rotshild and E.~M. Stein \cite{rs} and developed in many
other papers (\cite{good,hormel,cnsw, jean, bram1} etc.). We
present this construction in the form suitable for our purposes,
making essentially a synthesis of the ideas of papers \cite{cnsw}
and \cite{jean}, in order to get a (quasi)metric-decreasing
embedding of our C-C space into a regular one.

Using this embedding and results  for regular quasimetric C-C
spaces \cite{vk,vk1} we will derive some important geometric
properties of the quasimetrics $\rho$ and $\rho^u$, in particular
prove the generalized triangle inequality for both of them.
Crucial for proving main theorems of the next section is the
``rolling-of the-box lemma'' (Proposition \ref{prokat}).

Let us recall the construction of a free nilpotent Lie algebra ${\cal
N}_{d_1,\ldots,d_q}^M$ with $q$ generators ${\cal X}_1,\ldots,{\cal
X}_q$ of weights $\{d_i\}_{i=1}^q$ and depth $M$ \cite{cnsw}.

Let ${\cal F}_q$ be a free (infinite-dimensional) Lie algebra with
$q$ generators, i. e. the only interrelation between commutators of
vector fields $\{{\cal X}_i\}$ are  the scewcommutativity and the
Jacobi idenity. Introduce on ${\cal F}_q$ dilations acting as
\begin{equation}\label{dil_n}\delta_{\varepsilon}
(\sum\limits_{j=1}^qc_j{\cal X}_j)=
\sum\limits_{j=1}^qc_j\varepsilon^{d_j}{\cal X}_j,\quad
\delta_{\varepsilon}({\cal X}_I)=\varepsilon^{|I|_h}{\cal
X}_I.\end{equation} Consider subspaces ${\cal F}_q^l$,
 invariant of order $l$
under dilations \eqref{dil_n}. Then ${\cal
F}_q=\bigoplus\limits_{l=1}^{\infty}{\cal F}_q^l.$ Let
\begin{equation}\label{al_svob} {\cal N}={\cal
N}_{d_1,\ldots,d_q}^M={\cal F}_q/I_M,\quad\text{ where
}I_M=\bigoplus\limits_{l>M}{\cal F}_q^l\end{equation} is an Lie
algebra ideal in ${\cal F}_q$. Note that ${\cal F}_q/I_M$
isomorphic to the direct sum $\bigoplus\limits_{l\leq M}{\cal
F}_q^l.$

Let $\psi:{\cal N}\to \bigoplus\limits_{l\leq M}{\cal F}_q^l $ be a Lie algebra isomorphism and $ X_j=\psi({\cal X}_j)$.
 Denote
\begin{equation}\label{dimN}\tilde{N}=\tilde{N}(d_1,\ldots,d_q,M)=
\operatorname{dim} {\cal N}_{d_1,\ldots,d_q}^M.\end{equation}

\begin{Definition} The vector fields $\tilde{X}_1,\tilde{X}_2,\ldots,
\tilde{X}_q$ on $\tilde{U}\subseteq\tilde{\mathbb M}$,
defining a filtration of the form \eqref{filtr}, are called {\it free up to the order $s$} at the point $u\in \tilde{U}$, if
$\operatorname{dim}H_s(u)=\tilde{N}(d_1,\ldots,d_q,s)$.\end{Definition}

\begin{Remark}[\rm\cite{rs,cnsw}]\label{svob_reg} If the vector fields
$\tilde{X}_1,\tilde{X}_2,\ldots, \tilde{X}_q$ on
$\tilde{U}\subseteq\tilde{\mathbb M}$ are free up to the order $M$
at the point $u\in\tilde{U}$, where $M$ is the depth of the C-C
space $\mathbb M$, then the point $u$ is regular.\end{Remark}

The proof of the next proposition follows the same lines as the
proof of a similar assertion in \cite{jean} for the case of smooth
vector fields meeting the H\"ormander's condition. We recall this
proof, since some of its details are needed below.

\begin{Proposition}\label{lift}Let all conditions of Definition
$\ref{cc-space}$ be satisfied and $\tilde{N}$ be the
dimension defined by \eqref{dimN} of the corresponding
free Lie algebra. Consider the manifold
$\tilde{\mathbb M}=\mathbb M\times\mathbb
R^{\tilde{N}-N}$ of the dimension $\tilde{N}$.
Then there are a neighborhood $\tilde{U}$ of the point $(u,0)$
 in $\tilde{\mathbb M}$,
a neighborhood $U$ of the point $u$, where $U\times\{0\}\subseteq\tilde{U}$, coordinates $(y,z)$ on $\tilde{U}$ and two systems of  $C^M$-smooth vector fields
\begin{equation}\label{lift_vf}
\tilde{X}_k(y,z)=X_k(y)+\sum\limits_{j=N+1}^{\tilde{N}}b_{kj}(y,z)\frac
{\partial} {\partial z_j}
\text{ and }\widehat{\tilde{X}}_k(y,z)=\widehat{X}_k^u(y)+\sum\limits_
{j=N+1}^{\tilde{N}} b_{kj}(y,z)\frac{\partial} {\partial z_j},\
\end{equation}
$k=1,2,\ldots,q$, defining a C-C structure of depth  $M$ on $\tilde{U}\subseteq\tilde{\mathbb M}$and, hence, free up to order $M$ on $\tilde{U}$. Here $b_{jk}(y,z)$ are polynomial functions on  $\tilde{U}$, such that the vector fields
$\sum\limits_{j=N+1}^{\tilde{N}} b_{kj}(y,z)\frac{\partial}
{\partial z_j}$ are homogeneous of order $-d_k$, $k=1,2,\ldots,q$.

All points of some neighborhood $\tilde{V}=\tilde{V}(\tilde{u})
\subseteq\tilde{U}$
 are regular.
\end{Proposition}
\begin{proof}


 Consider canonical vector fields $\{\widehat{\tilde{X^{\prime}}}_I\}_{|I|_h\leq M}\in C^{\infty}$ on $\mathbb R^{\tilde{N}}$ which generate the Lie algebra ${\cal N}$ defined in \eqref{al_svob} in such way that ${\cal
F}_l^q=\text{span}\{\widehat{\tilde{X}}_I^{\prime}\}_{|I|_h\leq l}$,
$\widehat{\tilde{X^{\prime}}}_{I_j}(0)=e_j$,
$j=1,\ldots,\tilde{N}$ \cite{pos,fs,lan,vk}.

By definition of a free algebra, there is a surjective homomorphism of nilpotent Lie algebras $\Psi:{\cal N}\to
\text{Lie}\{\widehat{X}^u_1,\widehat{X}^u_2,\ldots,\widehat{X}^u_q\}$
such that
$\Psi(\widehat{\tilde{X^{\prime}}}_I)=\widehat{X}^{u}_I$,
$|I|_h\leq M$.

Let $G=\operatorname{exp}({\cal N})(0)$ be the corresponding Lie group $\mathbb R^{\tilde{N}}$. Define the action of $G$ on $\mathbb M$ by means of the homomorphism
$\Psi$: for
$g=\text{exp}\left(\sum\limits_{j=1}^{\tilde{N}}c_j
\widehat{\tilde{X^{\prime}}}_{I_j}\right)(0)\in G,\ v\in U$
let
$$g(v)=\text{exp}(\sum\limits_{j=1}^{\tilde{N}}c_j
\Psi(\widehat{\tilde{X^{\prime}}}))(v)=
\text{exp}(\sum\limits_{j=1}^{\tilde{N}}c_j
\widehat{X}_{I_j}^u)(v).$$

The isotropy subgroup $H=\{g\in G\mid g(u)=u\}\subseteq G$ is connected and invariant under dilations
\begin{equation}\label{dil_lift}
\tilde{\delta}_{\varepsilon}(x_1,x_2,\ldots,x_{\tilde{N}})=
(\varepsilon^{|I_1|_h}x_1, \varepsilon^{|I_2|_h}x_2,\ldots,
\varepsilon^{|I_{\tilde{N}}|_h}x_N),\end{equation} due to homogeneity of the vector fields.
Moreover,
 \begin{equation}\label{H_alg}{\cal
H}=\operatorname{span}\left\{\sum\limits_j
c_j\widehat{\tilde{X^{\prime}}}_{I_j}\mid \sum\limits_j
c_j\widehat{X}^u_{I_j}(u)=0\right\}.\end{equation}

Denote
 by $\widehat{Z}_{N+1},\ldots,\widehat{Z}_{\tilde{N}}$
the basis of the subalgebra ${\cal H}$ consisting of vector fields homogeneous under dilations.

The mapping
$\varphi_u:G\to U\subseteq \mathbb M$ defined as
$\varphi_u(g)=g(u)$ induces a diffeomorfism from the homogeneous space
$G/H=\{Hg\mid g\in G\}$ onto the neighborhood $U$:
$\varphi_u(Hg)=g(u).$ Consider on  $G/H$ left-invariant vector fields
$$\widehat{X}_i^h(Hg)=\frac{d}{dt}[Hg\operatorname{exp}
(t\widehat{\tilde{X^{\prime}}}_i)(0)]\Big\vert_{t=0}, i
=1,\ldots,q.$$ By the diffeomorphism $\varphi_u$ identify them with the vector fields $\widehat{X}_i^u$:
$$(\varphi_u)_*\langle\widehat{X}_i^h\rangle(Hg)=\frac{d}{dt}
[\varphi(
Hg\operatorname{exp}(t\widehat{\tilde{X^{\prime}}}_i)(0))]
\Big\vert_{t=0}=$$
$$=\frac{d}{dt}[\operatorname{exp}(t\widehat{X}^u_i)(g(u))]\Big\vert_{t=0}
= \widehat{X}^u_i(g(u)).$$

Consider on $U$ the basis
\begin{equation}\label{bas_hat}
\widehat{Y}^u_1,\widehat{Y}^u_2,\ldots,\widehat{Y}^u_N,\end{equation}
consisting of the same commutators of the vector fields
$\{\widehat{X}_I^u\}_{|I|_h\leq M}$, as the basis \eqref{bas} of the commutators of  $\{X_I\}_{|I|_h\leq M}$

Taking in account \eqref{H_alg}, we see that the family of vector
fields $\{\widehat{Y}_i\}_{i=1}=\{(\varphi_u)^{-1}_*
\langle\widehat{Y}^u_i\rangle\}_{i=1}^N$ is a basis of the
algebraic complement to $\mathcal{H}$ in the Lie subalgebra
$N_{M,m}$, consisting of homogeneous vector fields.

Introduce on  $G$ coordinates
\begin{equation}\label{odnor_coord}(y,z)\in\mathbb R^{\tilde{N}}\mapsto g=
\operatorname{exp}\left(\sum\limits_
{k=N+1}^{\tilde{N}}z_k\widehat{Z}_k\right)\operatorname{exp}
(y_{N}\widehat{Y}_{N})\ldots\operatorname{exp}
(y_{1}\widehat{Y}_{1})\end{equation}

In these coordinates it holds
\begin{equation}\label{decomp_group}
\widehat{\tilde{X^{\prime}}}_k(y,z)=\widehat{X}_k^h(y)+\sum
\limits_{j=N+1}^{\tilde{N}} b_{kj}(y,z)\frac{\partial} {\partial
z_j},\ j=1,2,\ldots,q.
\end{equation}
Indeed,
$$\widehat{\tilde{X^{\prime}}}_i(g)=\frac{d}{dt}[g\operatorname{exp}
(t\widehat{\tilde{X^{\prime}}}_i)](0) \mid_{t=0};$$ in coordinates
\eqref{odnor_coord} we have
$$g\operatorname{exp}(t\tilde{X}^{\prime}_i)(0)=
\operatorname{exp}\left(\sum\limits_
{k=N+1}^{\tilde{N}}z_k\widehat{Z}_k\right)\operatorname{exp}
(y_{N}\widehat{Y}_{N})\ldots\operatorname{exp}
(y_{1}\widehat{Y}_{1})
\operatorname{exp}(t\widehat{\tilde{X^{\prime}}}_i)(0)=$$
$$=
\operatorname{exp}\left(\sum\limits_
{k=N+1}^{\tilde{N}}z_k\widehat{Z}_k\right) h(t)\operatorname{exp}
(c_N(y,t)\widehat{Y}_{N})\ldots\operatorname{exp}
(c_1(y,t)\widehat{Y}_{1}),$$ where  $h(t)\in H$;
$$\hspace{-40.pt}Hg\operatorname{exp}(t\widehat{\tilde{X^{\prime}}}_i)(0)=
H\operatorname{exp}\left(\sum\limits_
{k=N+1}^{\tilde{N}}z_k\widehat{Z}_k\right) \operatorname{exp}
(y_{N}\widehat{Y}_{N})\ldots\operatorname{exp}
(y_{1}\widehat{Y}_{1})
\operatorname{exp}(t\widehat{\tilde{X^{\prime}}}_i)(0)=$$
$$=
H\operatorname{exp}
(c_{N}(y,t)\widehat{Y}_{N})\ldots\operatorname{exp}
(c_{1}(y,t)\widehat{Y}_{1}).$$ Thus the coordinates of the vector
fields $\widehat{X}^h_i$ and $\widehat{\tilde{X^{\prime}}}_i$ by
$\frac{\partial}{\partial y_k}$ coincide and are equal to
$\frac{d}{dt}c_k(y,0)$. Hence, we have \eqref{decomp_group}.

Now define the vector fields $\tilde{X}_k$,
$\widehat{\tilde{X}}_k$ by formulas \eqref{lift_vf}. Since the
vector fields $\widehat{\tilde{X^{\prime}}}_k$ are homogeneous of
order $-d_k$, then the vector fields
$\sum\limits_{j=N+1}^{\tilde{N}} b_{kj}(y,z)\frac {\partial}
{\partial z_j},\ k=1,2,\ldots,q$ are homogeneous of the same
order. By construction, we have that
$\widehat{\tilde{X}}_k=\tilde{X}_k^{(-d_k)}$ w.r.t. the dilations
\eqref{dil_lift}. Thus the vector fields $\{\tilde{X}_k\}_{k=1}^q$
define a C-C structure of depth  $M$ on $\tilde{\mathbb M}$ and
are free of order $M$ on $U$. The point  $\tilde{u}$ and hence all
points in some of its neighborhoods are regular, according to
Remark \ref{svob_reg}.
\end{proof}

\begin{Proposition}\label{lift_I}
For all multiindices $I$, such that $|I|_h\leq M$, the following decompositions hold:
\begin{equation}\label{lift_vf_I}
\tilde{X}_I(y,z)=X_I(y)+\sum\limits_{j=N+1}^{\tilde{N}}b_{Ij}(y,z)\frac
{\partial} {\partial z_j}\text{ and }
\widehat{\tilde{X}}_I(y,z)=\widehat{X}_I^u(y)+
\sum\limits_{j=N+1}^{\tilde{N}}
\widehat{b}_{Ij}(y,z)\frac{\partial} {\partial z_j},
\end{equation}
where $b_{Ij}(y,z),\widehat{b}_{Ij}(y,z)\in C^{M+1}(\tilde{U})$.
\end{Proposition}
\begin{proof}
Let us prove the first decomposition of \eqref{lift_vf_I} by
induction on the length of $I$ (the second decomposition is proved
in a similar way). Let \eqref{lift_vf_I} be true for all $J$, such
that  $|J|_h\leq l$. By the Jacobi identity, any vector field
$\tilde{X}_I$, where $|I|_h\leq l+\min\{d_1,\ldots,d_q\}$, can be
represented as $\tilde{X}_I=[\tilde{X}_i,\tilde{X}_J]$, where
$i\in{1,\ldots,q}$ and $|J|_h\leq l$. By induction and taking into
account the identity \eqref{lift_vf}, we get
$$\tilde{X}_I(y,x)=[X_i,X_J](y)+[X_i,
\sum\limits_{j=N+1}^{\tilde{N}} b_{Jj}(y,z)\frac{\partial}
{\partial z_j}]+$$$$+[\sum\limits_{j=N+1}^{\tilde{N}}
b_{ij}(y,z)\frac{\partial} {\partial
z_j},X_J]+[\sum\limits_{j=N+1}^{\tilde{N}}
b_{ij}(y,z)\frac{\partial} {\partial
z_j},\sum\limits_{j=N+1}^{\tilde{N}} b_{Jj}(y,z)\frac{\partial}
{\partial z_j}]$$
$$=X_I(y)+\sum\limits_{j=N+1}^{\tilde{N}} \left(X_ib_{Jj}-X_Jb_{ij}+
\frac{\partial} {\partial z_j}b_{Ji}-\frac{\partial} {\partial
z_j}b_{ij}\right)\frac{\partial} {\partial z_j}.$$ Thus the vector field
 $\tilde{X}_I$ has the desired form. The rest of the proposition
 follows from the smoothness assumptions of Definition  \ref{cc-space}.
\end{proof}

Consider the neighborhood $\tilde{U}$ and the vector fields
$\tilde{X}_I$ from Propositions \ref{lift}, \ref{lift_I}. Let
$\pi:\tilde{U}\to U$ be a {\it canonical projection} acting on an
arbitrary point $\tilde{v}=(v,y)$, such that $v\in U$,
$y\in\mathbb R^{\tilde{N}-N}$, as $\pi(\tilde{v})=v$.
The next proposition states that the projection is distance-decreasing (cf. \cite{bel,jean})
\begin{Proposition}\label{ineq_prop} For any $v,w\in U$ and
$p,q\in\mathbb R^{\tilde{N}-N}$ the following inequalities hold:
\begin{equation}\label{ineq_rho}\rho(v,w)\leq \tilde{\rho}((v,p),(w,q)),\end{equation}
 \begin{equation}\label{ineq_u_rho}\rho^u(v,w)\leq
\tilde{\rho}^{\tilde{u}}((v,p),(w,q)),\end{equation} where the
quasimetrics $\tilde{\rho}$, $\tilde{\rho}^{\tilde{u}}$ on the
regular C-C space  $\tilde{U}$ are defined in a similar  way as $\rho,\rho^u$ on
the initial neighborhood $U\subseteq\mathbb M$.
\end{Proposition}
\begin{proof} Show the inequality \eqref{ineq_rho}.
Denote $\tilde{v}=(v,p)$, $\tilde{w}=(w,q)$. There is a unique
curve $\tilde{\gamma}(t)$ such that
$$\begin{cases}\dot{\tilde{\gamma}}(t)=\sum\limits_{|I|_h\leq M}w_I
\tilde{X}_I(\tilde{\gamma}(t))=\sum\limits_{|I|_h\leq M}w_I
(X_I(\pi(\tilde{\gamma}(t)))+\sum\limits_{k=N+1}^{\tilde{N}}b_{Ik}
(\tilde{\gamma}(t))\frac{\partial}{\partial z_k}),\\
\tilde{\gamma}(0)=\tilde{v}, \tilde{\gamma}(1)=\tilde{w}.
\end{cases}$$

By definition,
$\tilde{\rho}(\tilde{v},\tilde{w})=\underset{|I|_h\leq
M}\max\{|w_I|^{1/|I|_h}\}$. Let
$\gamma(t)=\pi(\tilde{\gamma}(t))$, then
$$\begin{cases}\dot{\gamma}(t)=\sum\limits_{|I|_h\leq M}w_I
X_I(\gamma(t)),\\
\gamma(0)=v, \gamma(1)=w.
\end{cases}$$

Thus, the curve $\gamma(t)$  lies in $U$ and joins the points $v$
and $w$, from where \eqref{ineq_rho} follows:
$\rho(v,w)\leq\underset{|I|_h\leq
M}\max\{|w_I|^{1/|I|_h}\}=\tilde{\rho}(\tilde{v},\tilde{w})$. The
inequality \eqref{ineq_u_rho} is proved in the same way.
\end{proof}

\begin{Proposition}[\rm Generalized triangle inequalities]
\label{treug_ineq} For any point $g\in U$ there are constants
$Q,Q_g>0$ such that, for all $u,v,w\in U$, we have
\begin{equation}\label{treug_rho}\rho(v,w)\leq
Q(\rho(u,v)+\rho(u,w)),\end{equation}
\begin{equation}\label{treug_rhou}
\rho^g(v,w)\leq Q_g(\rho^g(u,v)+\rho^g(u,w)).\end{equation}
\end{Proposition}
\begin{proof}
For any (arbitrarily small) $\zeta>0$ consider
$$\{a_I\}_{|I|_h\leq M}\text{ and }\{b_I\}_{|I|_h\leq M},$$ such that
$$v=\text{exp}(\sum\limits_{|I|_h\leq M}a_IX_I)(u),\
w=\text{exp}(\sum\limits_{|I|_h\leq M}b_IX_I)(u)$$ and
$$\underset{|I|_h\leq
M}\max\{|a_I|^{1/|I|_h}\}\leq\rho(u,v)+\zeta,\ \underset{|I|_h\leq
M}\max\{|b_I|^{1/|I|_h}\}\leq\rho(u,w)+\zeta.$$ Let
$\tilde{u}=(u,0)$ and consider on $\tilde{U}$ points
$$\tilde{v}=\text{exp}(\sum\limits_{|I|_h\leq
M}a_I\tilde{X}_I)(\tilde{u})\text{ and }
\tilde{w}=\text{exp}(\sum\limits_{|I|_h\leq
M}b_I\tilde{X}_I)(\tilde{u}).$$ Then we have $v=\pi(\tilde{v}),
w=\pi(\tilde{w})$ and
$$\tilde{\rho}(\tilde{u},\tilde{v})=\underset{|I|_h\leq
M}\max\{|a_I|^{1/|I|_h}\},\
\tilde{\rho}(\tilde{u},\tilde{w})=\underset{|I|_h\leq
M}\max\{|b_I|^{1/|I|_h}\}.$$ According to Proposition \ref{ineq_prop} and the
generalized triangle inequality for $\tilde{\rho}$ (in the
neighborhood of a regular point \cite{vk}) we have
$$\rho(v,w)\leq \tilde{\rho}(\tilde{v},\tilde{w})\leq
Q(\tilde{\rho}(\tilde{u},\tilde{v})+\tilde{\rho}(\tilde{u},\tilde{w}))\leq
Q(\rho(u,v)+\rho(u,w)+2\zeta),$$ from where \eqref{treug_rho}
follows; \eqref{treug_rhou} is proved in a similar way.
\end{proof}

\begin{Proposition}[\rm ``Rolling-of-the-box'' lemma]\label{prokat}
For all points $u,v\in U$ and $r,\xi>0$, for which both parts of the
following inclusions make sense (i.e. lie in  $U$), we have
\begin{equation}\label{prokat_incl_u}
\underset{x\in B^{\rho^u}(v,r)}\bigcup B^{\rho^u}(x,\xi)\subseteq
B^{\rho^u}(v,r+C\xi),\end{equation}
\begin{equation}\label{prokat_incl} \underset{x\in
B^{\rho}(v,r)}\bigcup B^{\rho}(x,\xi)\subseteq
B^{\rho}(v,r+C\xi+O(r^{1+\frac{1}{M}})+O(\xi^{1+\frac{1}{M}})).
\end{equation}
\end{Proposition}
\begin{proof}
Let us prove \eqref{prokat_incl}. Fix points $x,z$, such that $\rho(v,x)< r$, $\rho(x,z)<
\xi$, and show that $\rho(v,z)<
r+C\xi+O(r^{1+\frac{1}{M}})+O(\xi^{1+\frac{1}{M}})$. For arbitrarily
small
 $\zeta>0$ 
consider two curves $\gamma_1$, $\gamma_2$, such that
$$\begin{cases}\dot{\gamma_1}(t)=\sum\limits_{|I|_h\leq M}x_I
X_I(\gamma_1(t)),\\
\gamma_1(0)=v, \gamma_1(1)=x,
\end{cases}\quad\quad
\begin{cases}\dot{\gamma_2}(t)=\sum\limits_{|I|_h\leq M}z_I
X_I(\gamma_2(t)),\\
\gamma_2(0)=x, \gamma_2(1)=z,
\end{cases}
$$
and $$\underset{|I|_h\leq
M}\max\{|x_I|^{1/|I|_h}\}\leq\rho(v,x)+\zeta,\ \underset{|I|_h\leq
M}\max\{|z_I|^{1/|I|_h}\}\leq\rho(x,z)+\zeta.$$

Consider a point $\tilde{v}=(v,0)\in U$ and a curve
$\tilde{\gamma}_1$ such that
$$\begin{cases}\dot{\tilde{\gamma}}_1(t)=\sum\limits_{|I|_h\leq M}x_I
\tilde{X}_I(\tilde{\gamma}_1(t)),\\
\tilde{\gamma}_1(0)=\tilde{v}.\end{cases}$$ Since
$\gamma_1(t)=\pi(\tilde{\gamma}_1(t))$, we have
$\tilde{\gamma}_1(1)=(x,p)=:\tilde{x}\in\tilde{U}$, where $p\in\mathbb
R^{\tilde{N}-N}$. However,
$$\tilde{\rho}(\tilde{v},\tilde{x})=\underset{|I|_h\leq
M}\max\{|x_I|^{1/|I|_h}\}<r+\zeta.$$
In a similar way, for  a curve $\tilde{\gamma}_2$, such that
$$\begin{cases}\dot{\tilde{\gamma}}_2(t)=\sum\limits_{|I|_h\leq M}x_I
\tilde{X}_I(\tilde{\gamma}_2(t)),\\
\tilde{\gamma}_2(0)=\tilde{x}.\end{cases}$$ We have
$\gamma_2(t)=\pi(\tilde{\gamma}_2(t))$, and hence
$\tilde{\gamma}_2(1)=(z,q)=:\tilde{z}\in\tilde{U}$, where $q\in\mathbb
R^{\tilde{N}-N}$, and
$$\tilde{\rho}(\tilde{x},\tilde{z})=\underset{|I|_h\leq
M}\max\{|z_I|^{1/|I|_h}\}<\xi+\zeta.$$

According to Remark \ref{svob_reg}, all points of $\tilde{U}$ are
regular w. r. t. the C-C structure induced by the vector fields $\{\tilde{X}_I\}_{|I|_h}\leq M$.

By the Campbell-Hausdorff formula \cite{lan}, for any vector fields
$X,Y\in C^{k_0+1}$ the following decomposition is true:
\begin{equation}\label{kh}\text{exp}(sY)\circ\text{exp}(tX)(v)=
\text{exp}(sY+tX+\frac{st}{2}[X,Y]+\end{equation}
$$+\sum\limits_{2\leq k+j\leq
k_0}s^kt^jC_{kj}(X,Y)+O(s^{k_0+1})+O(t^{k_0+1})),$$ where
$C_{kj}(X,Y)$ are linear combinations of ($k+j-1$)-order commutators of $X$ and $Y$.

Applying \eqref{kh}, by simple computations, we get
$$\operatorname{exp}\left(\sum\limits_{|I|_h\leq M}z_I\tilde{X}_I
\right)\circ\operatorname{exp}\left(\sum\limits_{|I|_h\leq
M}x_I\tilde{X}_I
\right)=\operatorname{exp}\left(\sum\limits_{|I|_h\leq
M}v_I\tilde{X}_I \right)(v),$$ where
$$v_I=x_I+y_I+\sum\limits_{\substack {|\alpha+\beta|\leq M\\
|\alpha+\beta|_h\geq|I|_h}}F^I_{\alpha,\beta}x^{\alpha}z^{\beta}+
O(||x||^{M+1}) +O(||z||^{M+1}).$$

Consequently,
$$|v_I|\leq|x_I|+|z_I|+\sum\limits_{|\alpha+\beta|_h=|I|_h}
|F^I_{\alpha,\beta}|x^{\alpha}z^{\beta}+$$$$+\sum
\limits_{\substack{|\alpha+\beta|\leq M\\|\alpha+\beta|_h>|I|_h}}
|F^I_{\alpha,\beta}|x^{\alpha}z^{\beta}+O(||x||^{M+1})
+O(||z||^{M+1})\leq$$$$\leq(\tilde{r}+C_I\tilde{\xi})^{|I|_h}+O(\tilde{r}
^{|I|_h+1}) +O(\tilde{\xi}^{|I|_h+1})+O(\tilde{r}^{M+1})+
O(\tilde{\xi}^{M+1}),$$ from where it follows, that
$$\tilde{\rho}(\tilde{v},\tilde{z})=\underset{|I|_h\leq
M}\max\{|v_I|^{1/|I|_h}\}\leq
\tilde{r}+C\tilde{\xi}+O(\tilde{r}^{1+\frac{1}{M}})+
O(\tilde{\xi}^{1+\frac{1}{M}}).$$
Applying \eqref{ineq_rho}, we finally obtain
$$\rho(v,z)\leq \tilde{\rho}(\tilde{v},\tilde{z})\leq r
+C\xi+O(r^{1+\frac{1}{M}})+O(\xi^{1+\frac{1}{M}})+O(\zeta),$$ from
where \eqref{prokat_incl} follows. The inclusion \eqref{prokat_incl_u} can be proved in a similar way.
\end{proof}

\section{Main theorems on local geometry}

\begin{Proposition}
Consider on $\tilde{U}$ bases $\{\tilde{X}_I\}_{|I|_h\leq M}\text{
and }
 \{\widehat{\tilde{X}}_{I}\}_{|I|_h\leq M},$ consisting of commutators
 of the vector fields defined in
 \eqref{lift_vf}.

 Then, in coordinates $x=(y,z)$ defined in \eqref{odnor_coord},
 for all $x\in\tilde{U}$, such that $|x_j|\leq
 \varepsilon^{|I_j|}$, the following decompositions hold:
  \begin{equation}\label{dec1}
 \tilde{X}_I(x)=\sum\limits_{i=1}^{\tilde{N}}a_{I,J}(x)
 \widehat{\tilde{X}}_J(x),
 \end{equation}
 where
 $$ a_{I,J}=
 \begin{cases}
\delta_{I,J}+O(\varepsilon),\quad&|J|_h=|I|_h,\\
o(\varepsilon^{|J|_h-|I|_h}),\quad&|J|_h>|I|_h,\\
O(1),\quad&|J|_h<|I|_h.\\
 \end{cases}
$$
 \end{Proposition}
\begin{proof}
From Propositions \ref{lift_I} and \ref{vf_decomp} it follows that
$$\tilde{X}_{I}(x)=\widehat{\tilde{X}}_{I}(x)+R_I(x),$$ where
$x=(y,z)\in\mathbb R^{\tilde{N}}$, while the vector field $R_I$
consists of summands of  homogeneity order, w.r.t. the dilations
\eqref{dil_lift}, bigger than
 $-|I|_h$. Since the vector fields $\widehat{\tilde{X}}_J$
 are homogeneous of order $|J|_h$, we have
$$R_I(x)=\sum\limits_{|J|_h\leq M}\sum\limits_{|\alpha|_h>
|J|_h-|I|_h} c_{Il}x^{\alpha}\widehat{\tilde{X}}_J=$$
$$=\sum\limits_{|J|_h>|I|_h}\varepsilon^{|J|_h-|I|_h+1}
(O(1)+O(\varepsilon))\widehat{\tilde{X}}_J+$$$$+\sum\limits_{|J|_h=|I|_h}
\varepsilon(O(1)+O(\varepsilon))\widehat{\tilde{X}}_J+
\sum\limits_{|J|_h<|I|_h}(O(1)+O(\varepsilon))\widehat{\tilde{X}}_J=\sum
a_{I,J}\widehat{\tilde{X}}_J,$$

from where the proposition follows.


\end{proof}

Next we introduce an important characteristic of the C-C space
$\mathbb M$.

\begin{Definition}\label{int_lines_koef}
 Let $u,v\in U$, $r>0$.  {\it The divergence of integral lines}
with nilpotentizations centered at $u$ over a box of radius $r$
centered at $v$ is the value
\begin{equation}\label{R}
R(u,v,r)=\max\{\underset{\widehat{y}\in
B^{\rho^u}(v,r)}\sup\{\rho^u(y, \widehat{y})\},\underset{y\in
B^{\rho}(v,r)}\sup\{\rho(y, \widehat{y})\}\}.\end{equation} Here
the points $y$ and $\widehat{y}$ are defined as follows. Let
$\gamma(t)$ be an arbitrary curve, defined as a solution of the
system of ODE
$$\begin{cases}\dot{\gamma}(t)=\sum\limits_{|I|_h\leq M}b_I
\widehat{X}^u_I(\gamma(t)),\\
\gamma(0)=v, \gamma(1)=\widehat{y},
\end{cases}$$
and
\begin{equation}\label{b_I}\rho^u(v,\widehat{y})\leq\underset{|I|_h\leq
M}\max\{|b_I|^{1/|I|_h}\}\leq r.\end{equation} Define
$y=\text{exp}(\sum\limits_{|I|_h\leq M}b_I X_I)(v)$. In this way,
the supremum in the first expression of \eqref{R} is taken not
only over $\widehat{y}\in B^{\rho^u}(v,r)$, but also over the
infinite set of the possible $\{b_I\}_{|I|_h\leq M}$, satisfying
\eqref{b_I}. The
second expression is understood in a similar way.
\end{Definition}

\begin{Proposition}\label{lemma_incl}
Let $u,v\in U$ and $r>0$. Then the following inclusions are true:
\begin{equation}\label{incl1}B^{\rho}(v,r)
\subseteq B^{\rho^u}(v,r+CR(u,v,r)),\end{equation}
\begin{equation}\label{incl2}B^{\rho^u}(v,r)\subseteq
B^{\rho}(v,r+CR(u,v,r)+O(r^{1+\frac{1}{M}})+O(R(u,v,r)^{1+
\frac{1}{M}})),\end{equation}
 where $R(u,v,r)$ is defined by
\eqref{R}.
\end{Proposition}
\begin{proof}
Let $y\in B^{\rho}(v,r)$, i.e. $\rho(v,y)<r$, and show that
$\rho^u(v,y)<r+CR(u,v,r)$ for some constant $C$.

By definition of the quasimetric $\rho$, for arbitrarily small
$\zeta>0$ there are $\{a_I\}_{|I|_h\leq M}$, such that
$$y=\text{exp}(\sum\limits_{|I|_h\leq M}a_IX_I)(v)$$ and
$$\underset{|I|_h\leq
M}\max\{|a_I|^{1/|I|_h}\}\leq\rho(v,y)+\zeta\leq r.$$
 Consider a point
$\widehat{y}=\text{exp}(\sum\limits_{|I|_h\leq M}a_IX_I)(v).$ Then
$\widehat{y}\in B^{\rho^u}(v,r)$, since
$$\rho^u(v,\widehat{y})\leq\underset{|I|_h\leq
M}\max\{|a_I|^{1/|I|_h}\}\leq r.$$ Obviously,
$\rho^u(y,\widehat{y})<R(u,v,r)$. Hence, by \eqref{prokat_incl_u},
$$y\in\underset{x\in B^{\rho^u}(v,r)}\bigcup B^{\rho^u}(x,R(u,v,r))
\subseteq B^{\rho^u}(v,r+CR(u,v,r)),$$ and \eqref{incl1} is
proved.

The inclusion \eqref{incl2} is proved in the same way with the
application of \eqref{prokat_incl}.
\end{proof}

\begin{Theorem}[\rm Theorem on divergence of integral lines]
\label{theorem_int_lines} Let $u,v\in U$,
$\rho(u,v)=O(\varepsilon)$, $r=O(\varepsilon)$ and
$B^{\rho}(v,r)\cup B^{\rho^u}(v,r)\subseteq U$. Then we have the
following estimate on divergence of integral lines from Definition
$\ref{int_lines_koef}$:
$$R(u,v,r)=O(\varepsilon^{1+\frac{1}{M}}).$$
\end{Theorem}
\begin{proof}
For a fixed point $\widehat{y}\in B^{\rho^u}(v,r)$ and $\zeta>0$ we consider
arbitrary $\{b_I\}_{|I|_h\leq M}$ such that
$$\widehat{y}=\operatorname{exp}(\sum\limits_{|I|_h\leq
M}b_I\widehat{X}^u_I)(v)\quad\text{
 and
}\quad\underset{|I|_h\leq
M}\max\{|b_I|^{1/|I|_h}\}\leq\rho^u(v,\widehat{y})+\zeta\leq r.$$ Let
$y=\operatorname{exp}\bigl(\sum\limits_{|I|_h\leq M}b_IX_I\bigr)(v)$ and
 $v=\text{exp}\Bigl(\sum\limits_{|I|_h\leq M}v_IX_I\Bigr)(u)\in
U$. Consider points
$$\tilde{v}=\text{exp}\Bigl(\sum\limits_{|I|_h\leq
M}v_I\tilde{X}_I\Bigr)(u,0)\in \tilde{U}\quad\text{
 and
}\quad\tilde{y}=\text{exp}\Bigl(\sum\limits_{|I|_h\leq
M}b_I\tilde{X}_I\Bigr)(\tilde{v})\in\tilde{U}.$$ Then
$$\tilde{\rho}(\tilde{v},\tilde{y})=\underset{|I|_h\leq
M}\max\{|b_I|^{1/|I|_h}\}=O(\varepsilon).$$ Let
$\widehat{\tilde{y}}:=\text{exp}(\sum\limits_{|I|_h\leq
M}b_I\tilde{X_I})(v)$. Since all points of $\tilde{U}$ are
regular, from Theorem \ref{int_lines_perem} it follows that
$$\max\{\tilde{\rho}(\tilde{y},\widehat{\tilde{y}}),
\tilde{\rho}^{\tilde{u}} (\tilde{y},\widehat{\tilde{y}})\} =
O(\varepsilon^{1+\frac{1}{M}}),$$ from where, taking into account
Proposition 8, the proposition follows. The application of this
theorem is possible due to Proposition 11.
\end{proof}

\begin{Remark}
In the paper  \cite{vk1}, where Theorem \ref{int_lines_perem} was
proved,  the nilpotentized vector fields satisfy estimates
\eqref{dec} which are stronger than \eqref{dec1}, namely, with
$O(\varepsilon)$ in place of $O(1)$ in the last estimate. Here we
can not guarantee $O(\varepsilon)$ because, in contrast to the
case of regular points, not all of the values of commutators
$\widehat{X}_I(u)$ at $u$ might coincide with the values $X_I(u)$
(see \cite{herm} and references therein). Nevertheless, a revision
of the proof of Theorem \ref{int_lines_perem} shows that it holds
also with these weaker estimates. Note also that this theorem
\ref{int_lines_perem} is true in any coordinates, in  which the
decomposition \eqref{dec} or \eqref{dec1} is true.
\end{Remark}

\begin{Theorem}[\rm Local approximation theorem]\label{lat_quasim}
For any points $u\in U$ and $v,w\in U$, such
that $\rho(u,v)=O(\varepsilon)$, $\rho(u,w)=O(\varepsilon)$, we
have
$$|\rho(v,w)-\rho^u(v,w)|=O(\varepsilon^{1+\frac{1}{M}}).$$
\end{Theorem}

\begin{proof}
 In Proposition
\ref{lemma_incl} let $r:=\rho(v,w)$. Then
$w\in\bar{B}^{\rho}(v,r)$, hence
$$\rho^u(v,w)\leq\rho(v,w)+CR(u,v,r).$$ In the same way, setting
$r:=\rho^u(v,w)$, we obtain
$$\rho(v,w)\leq\rho^u(v,w)+CR(u,v,r)+O(r^{1+\frac{1}{M}})+
O(R(u,v,r)^{1+\frac{1}{M}}).$$ Due to Proposition \ref{treug}
(generalized triangle inequality for $\rho$) we have
$r=O(\varepsilon)$, since from Theorem \ref{theorem_int_lines} the
proposition follows.
\end{proof}


\section{The tangent cone theorems}

First we briefly recall the notion and basic properties of
convergence of a sequence of quasimetric spaces, as well as the
notion of the tangent cone to a quasimetric space, introduced in
\cite{dan,smj} as an extension of Gromov's theory for metric
spaces.

The {\it distortion} (see e.g. \cite{bur}) of a mapping
$f:(X,d_X)\to (Y,d_Y)$ is the
value
$$\operatorname{dis}(f)=\sup\limits_{u,v\in
X}|d_Y(f(u),f(v))-d_X(u,v)|,$$ which is a measure of difference of
$f$ from an isometry.

\begin{Definition}[\cite{dan,smj}]\label{ghq-dist}
The {\it distance} $d_{qm}(X,Y)$ between quasimetric spaces
\linebreak $(X,d_X)$ and $(Y,d_Y)$ is defined as the infimum taken
over $\rho>0$ for which there exist (not necessarily continuous)
mappings $f:X\to Y$ and $g:Y\to X$ such that
$$\max\Bigl\{\operatorname{dis}(f),\,
\operatorname{dis}(g),\,\sup\limits_{x\in X}d_{X}(x,g(f(x))),\,
\sup\limits_{y\in Y}d_{Y}(y,f(g(y)))\Bigr\}\leq\rho.
$$
\end{Definition}

Note that for bounded quasimetric spaces the introduced distance is
obviously finite.

\begin{Proposition}\label{s1} The distance $d_{qm}$ possesses the following
properties$:$

$1)$ if quasimetric spaces $X$ and $Y$ are isometric, then
$d_{qm}(X,Y)=0$; if $X$ and $Y$ are compact and $d_{qm}(X,Y)=0$,
then $X$ and $Y$ are isometric $($nondegeneracy$)$.

$2)$ $d_{qm}(X,Y)=d_{qm}(Y,X)$ $($symmetry$)$.

$3)$ $d_{qm}(X,Y)\leq(Q_Z+1)(d_{qm}(X,Z)+d_{qm}(Z,Y))$ $($analog of
the generalized triangle inequality$)$.
\end{Proposition}

Note that the constant in 3) depends on the constant $Q_Z$.





By means of the (quasi)distance $d_{qm}$ a convergence, the limit
by which is unique up to isometry, for compact quasimetric spaces
can be introduced, in a similar way as it was done for metric
spaces. Namely, for a sequence $\{X_n\}$ of compact quasimetric
spaces, we say that $X_n\to X$, if $d_{qm}(X_n,X)\to 0$, when
$n\to\infty$. Note that a straightforward generalization  of
Gromov's definition of the distance $d_{GH}$ between two metric
spaces is possible only for a particular class of quasimetric
spaces \cite{gre1}.

 For noncompact spaces we use the following more
general notion of convergence. A {\it pointed $($quasi$)$metric
space} is a pair $(X,p)$ consisting of a (quasi)metric space $X$
and a point $p\in X$. Whenever we want to emphasize what kind of
(quasi)metric is on $X$, we shall write the pointed space as a
triple $(X,p,d_X)$.

\begin{Definition}\label{km_conv_nc}
A sequence $(X_n,p_n,d_{X_n})$ of pointed quasimetric spaces {\it
converges} to the pointed space $(X,p,d_X)$, if there exists a
sequence of reals $\delta_n\to 0$ such that for each $r>0$ there
exist mappings $f_{n,r}:B^{d_{X_n}}(p_n,r+\delta_n)\to X,\
g_{n,r}:B^{d_{X}}(p,r+2\delta_n)\to X_n$ such that

1) $f_{n,r}(p_n)=p, \ g_{n,r}(p)=p_n$;

2) $\operatorname{dis}(f_{n,r})<\delta_n,\
\operatorname{dis}(g_{n,r})<\delta_n;$

3) $\sup\limits_{x\in
B^{d_{X_n}}(p_n,r+\delta_n)}d_{X_n}(x,g_{n,r}(f_{n,r}(x)))<\delta_n$.
\end{Definition}

Recall that a quasimetric space $X$ is {\it boundedly compact}, if
all closed bounded subsets of $X$ are compact. Two pointed
quasimetric spaces $(X,p)$ and $(Y,q)$ are called {\it isometric},
if there exists an isometry $\eta:Y\to X$ such that $\eta(q)=p$. The
following theorem $($see \cite{dan,smj} for details$)$ informally
states
 that, for boundedly compact spaces, the limit is unique up to isometry.

\begin{Theorem}\label{ed_km}
1) Reduced to the case of metric spaces, the convergence of Definition \ref{km_conv_nc} is equivalent to the Gromov-Hausdorff convergence.

2) Let $(X,p),\ (Y,q)$ be two complete pointed quasimetric spaces
obtained as limits $($in the sense of definition \ref{km_conv_nc}$)$
of the same sequence $(X_n,p_n)$ such that $|Q_{X_n}|\leq C$ for all
$n\in\mathbb N$. If $X$ is boundedly compact then $(X,p)$ and
$(Y,q)$ are isometric.
\end{Theorem}

The tangent cone is then defined as usual:

\begin{Definition}\label{cone_def} Let $X$ be a boundedly compact
(quasi)metric space, $p\in X$.
If the limit of pointed spaces
$\lim\limits_{\lambda\to\infty}(\lambda X,p)=(T_pX,e)$ (in the sense
of definition \ref{km_conv_nc}) exists, then
 $T_p X$ is called the {\it tangent cone} to $X$ at $p$.
Here $\lambda X=(X,\lambda\cdot d_X)$; the symbol
$\lim\limits_{\lambda\to\infty} (\lambda X,p)$ means that, for any
sequence $\lambda_n\to\infty$, there exists
$\lim\limits_{\lambda_n\to\infty}(\lambda_n X,p)$ which is
 independent of the choice of sequence $\lambda_n\to\infty$ as
 $n\to\infty$.

 {\it A local tangent cone} is an arbitrary neighborhood
 $U(e)\subseteq T_pX$ of fixed point $e\in T_pX$.

\end{Definition}

\begin{Remark}According to Theorem \ref{ed_km},  the tangent cone from
Definition
\ref{cone_def} is unique up to
isometry, i.~e. one should treat it  as a
 class of pointed quasimetric spaces isometric to each other. Note also that
  the tangent cone is isometric to $(\lambda T_pX,e)$ for all
$\lambda>0$ and is completely defined by any (arbitrarily small) neighborhood
  of the point.
 \end{Remark}

\begin{Theorem}\label{cc_teor}
Let $\mathbb M$ be a C-C space from Definition \ref{cc-space}.
Then the quasimetric space  $(U,\rho^u)$ is a local tangent cone at the point $u$ to the quasimetric space $(U,\rho)$, where  the quasimetrics $\rho$ and $\rho^u$
are defined by  \eqref{rho} and \eqref{rho_u}, respectively. The tangent cone is a homogeneous space $G/H$, constructed in the proof of the Proposition
\ref{lift} (here $G$ is a nilpotent graded group).
\end{Theorem}

\begin{proof}
We have to verify Definition \ref{cone_def} for the spaces $X_n=(U,u,\lambda_n\cdot \rho)$, $X=(U,u,\rho^u)$, where
$\lambda_n\to\infty,\ \lambda_n\geq 0$ is an arbitrary sequence of reals (w.l.o.g. we assume
$\lambda_n\geq 1$). It is sufficient to take $$f_{n_r}=\Delta_{\lambda_n}^u,\
g_{n,r}=\Delta_{\lambda_n^{-1}}^u. $$ Due to the conical property
\eqref{cone_prop} and Theorem \ref{lat_quasim} we have the first assertion.

To verify the second assertion, we have to verify the left-invariance of $\rho^u$, i.e. to prove that
\begin{equation}\label{levoinv}\rho^u(g(v),g(w))=\rho^u(v,w),\end{equation}
  where  $g$ is defined in Proposition \ref{lift}.

Consider a curve $\gamma(t)$ such that
$$\begin{cases}\dot{\gamma}(t)=\sum\limits_{|I|_h\leq M}b_I
\widehat{X}^u_I(\gamma(t)),\\
\gamma(0)=v, \gamma(1)=w.
\end{cases}$$
Due to the left-invariance of the vector fields
$\{\widehat{\tilde{X^{\prime}}}^u_I\}_{|I|_h\leq M}$, introduced
in the proof of Proposition \ref{lift}, and the existence of the
homomorphism
$\Psi(\widehat{\tilde{X^{\prime}}}^u_I)=\widehat{X^{\prime}}^u_I$,
the curve $\gamma_g(t)=g(\gamma(t))$ is a solution of the system
of equations
$$\begin{cases}\dot{\gamma_g}(t)=\sum\limits_{|I|_h\leq M}b_I
\widehat{X}^u_I(\gamma_g(t)),\\
\gamma(0)=g(v), \gamma(1)=g(w).
\end{cases}$$
By definition of the quasimetric $\rho^u$, we get the required assertion.
\end{proof}

\begin{Corollary}At a regular point, the tangent cone to a weighted
C-C space is a nilpotent graded group.\end{Corollary}

\section{The case of H\"{o}rmander vector fields}

\begin{Definition}\label{horm_usl} The vector fields
$\{X_1,\ldots,X_m\}\in C^{p}$ on $U\subseteq\mathbb M$, $m\leq N$,
meet H\"ormander's condition of depth $M$, if they span, by their
commutators up to the order $M-1$, the whole tangent space
$T_u\mathbb M$ at any point $u\in U$, and $M$ is the minimal
number with such property.
 \end{Definition}

Obviously, for the case of regular points, $\mathbb M$ is an
example of a Carnot manifold, see Definition \ref{cc-manifold}. In
this paper we assume that $p=2M+1$.

The homogeneous degree of the vector field $X_I$ is now equal to
its commutator order
$$\text{deg}(X_I)=\text{degalg}(X_I)=|I|=i_1+\ldots+i_k,$$
and the conditions (ii) and (iii) for the basis \eqref{bas}
coincide. Introduce the same local coordinates on $U$ as in
\eqref{loc_coord} and construct the nilpotent approximations
$\{\widehat{X}_I^u\}_{|I|\leq M}$, as in Proposition
\ref{vf_decomp}. The lifting construction is also carried out in a
similar way as before, see Proposition \ref{lift}. Here we have
$q=m$ and the Lie group of the free algebra ${\cal N}$ is a Carnot
group. These constructions and results of \cite{vk} for regular
points allow to prove an analog of the Rashevsky-Chow theorem for
spaces from Definition \ref{horm_usl}. This result is, however,
not new, in particular, the existence of $d_c$ for the case when
$p=M-1,\alpha$ was proved in \cite{bram} with other methods.

\begin{Theorem}\label{chow_nereg}
On $U$ there are finite metrics
\begin{equation}\label{dcu}d_c(v,w)=
\inf\limits_{\substack{\dot{\gamma}\in
H\mathbb M\\
\gamma(0)=v,\gamma(1)=w}}\{L(\gamma)\}\quad\text{ and }\quad
d_c^u(v,w)=\inf\limits_{\substack{\dot{\widehat{\gamma}}\in
\widehat{H}^u\mathbb M  \\
\widehat{\gamma}(0)=v,\widehat{\gamma}(1)=w}}\{L(\widehat{\gamma})\}.
\end{equation}\end{Theorem}
\begin{proof}
Consider the manifold $\tilde{\mathbb M}$ and the vector fields
$\tilde{X}_i,\widehat{\tilde{X}}_i$ constructed in Proposition
\ref{lift}. Due to Remark \ref{svob_reg} and to the results of
\cite{vk} for regular points, on the neighborhood $\tilde{U}$
there are finite metrics $\tilde{d}_c$ and
$\tilde{d_c}^{\tilde{u}}$, defined by the horizontal vector fields
$\tilde{X}_i$ and $\widehat{\tilde{X}}_i$, respectively.

Denote as $\pi:\tilde{\mathbb M}\to\mathbb M$ the canonical
projection, i. e. $\pi(v,z)=v$, where $v\in\mathbb M$,
$z\in\mathbb R^{\tilde{N}-N}$.

Assume $\tilde{\gamma}(t):[0,1]\to\tilde{U}$ be a geodesic of the
distance $\tilde{d_c}((v,0),(w,0))$. Consider the curve
$\gamma:[0,1]\to U$ defined as $\gamma(t)=\pi
(\tilde{\gamma}(t))$. Then, in coordinates \eqref{loc_coord}, we
have
\begin{equation}\label{lift_eq}\begin{cases}
\dot{\tilde{\gamma}}(t)=\sum\limits_{i=1}^ma_i(t)
\tilde{X}_i(\tilde{\gamma}(t))=\sum\limits_{i=1}^ma_i(t)
\left[X_i(\gamma(t))
+\sum\limits_{i=N+1}^{\tilde{N}}b_{ij}(\tilde{\gamma}(t))
\frac{\partial}{\partial
z_j}\right],
\\
\tilde{\gamma}(0)=(v,0), \tilde{\gamma}(0)=(w,0),
\end{cases}
\end{equation}
hence the curve $\gamma(t)$ connects the points $v,w\in U$ and is
horizontal w. r. t. the vector fields $X_1,\ldots, X_m$.

The proof for \eqref{dcu} is carried out in a similar way, with
help of the existence of the metric $\tilde{d_c}^{\tilde{u}}$.
\end{proof}

Since the vector fields $\{\widehat{X}_i\}$ are homogeneous of
order $-1$, the metric \eqref{dcu} meets the conical property:
\begin{equation}\label{cone_dcu}d_c^u(\Delta_{\varepsilon}^uv,\Delta_
{\varepsilon}^uw)= \varepsilon d_c^u(v,w).\end{equation}

The next two propositions are proved in the same way as in the
``classical'' $C^{\infty}$-smooth case \cite{rs,bel,jean}; we
write down the proofs for the convenience of the reader.
\begin{Proposition}\label{proj}
The projections of the balls w. r. t. the metric $\tilde{d}_c$
onto the initial neighborhood $U\subseteq\mathbb M$ coincide with
the balls w. r. t. the metric $d_c$, i. e.
\begin{equation}\label{ball_proj}
B^{d_c}(v,r)=\pi\left(B^{\tilde d_c}((u,z),r)\right),
\end{equation}
where $u\in U$, $z\in\mathbb R^{\tilde{N}-N}$, $\pi:\tilde{\mathbb
M}\to\mathbb M$ is a canonical projection $\pi(v,z)=v$.
\end{Proposition}
\begin{proof}
Let $\tilde{\gamma}(t):[0,1]\to\tilde{U}$ be any horizontal curve
starting from $(v,z)$. Then
\begin{equation}\label{lift_eq}\begin{cases}
\dot{\tilde{\gamma}}(t)=\sum\limits_{i=1}^ma_i(t)
\xi_i(\tilde{\gamma}(t))=\sum\limits_{i=1}^ma_i(t)\left[X_i(\pi
(\tilde{\gamma}(t))+\sum\limits_{i=N+1}^{\tilde{N}}b_{ij}
(\tilde{\gamma}(t))\frac{\partial}{\partial z_j}\right],
\\
\tilde{\gamma}(0)=(v,z).
\end{cases}
\end{equation}
Denote
\begin{equation}\label{proj_curve}\gamma(t)=\pi(\tilde{\gamma}(t)),
\end{equation}
then
\begin{equation}\label{tilde}\tilde{\gamma}(t)=\left(\begin{array}{c}
\gamma(t)\\---
\\\tilde{\gamma}_{N+1}(t)\\
\tilde{\gamma}_{N+2}(t)\\ \ldots \\
\tilde{\gamma}_{N+p}(t)\end{array}\right)\end{equation} and
\begin{equation}\label{lift_eq}\begin{cases}
\dot{\gamma}(t)=\sum\limits_{i=1}^ma_i(t)X_i(\gamma(t)),\\
\gamma(0)=v,
\end{cases}
\end{equation}
i. e. the curve $\gamma(t)=\pi(\tilde{\gamma}(t))$ is horizontal
w. r. t. the vector fields $X_1,X_2,\ldots, X_m$ and is of the
same length as $\tilde{\gamma}(t)$, i. e. the projections of
horizontal curves on $\tilde{\mathbb M}$ are horizontal curves on
$\mathbb M$.

Conversely, if $\gamma(t)$ is a horizontal curve on $\mathbb M$, a
horizontal curve $\tilde{\gamma}(t)$ on $\tilde{\mathbb M}$ can be
defined in such way that \eqref{proj_curve} holds. Indeed, it is
sufficient to define $\tilde{\gamma}(t)$ by \eqref{tilde}, where
the last $\tilde{N}-N$ components are computed as the solutions of
the Cauchy problem
$$\begin{cases}\dot{\tilde{\gamma}}_{N+j}(t)=\sum\limits_{i=1}^m
a_i(t)b_{ij}(\tilde{\gamma}(t)),\\
\tilde{\gamma}_{N+j}(0)=z_j.\end{cases}$$ In this way, the set of
horizontal curves on $\mathbb M$ coincides with the set of
projections of the horizontal curves on $\tilde{\mathbb M}$, hence
the equality of balls \eqref{ball_proj} is true.
\end{proof}

\begin{Proposition}\label{ineq_prop} The projection $\pi$ is
distance-decreasing, i. e. for any points $v,w\in U$,
$p,q\in\mathbb R^{\tilde{N}-N}$ the following inequalities hold:
\begin{equation}\label{ineq}d_c(v,w)\leq \tilde{d}_c((v,p),(w,q)),
\end{equation}
\begin{equation}\label{ineq_u}d^u_c(v,w)\leq \tilde{d}_c^u((v,p),(w,q)).
\end{equation}
\end{Proposition}
\begin{proof}
Denote $\tilde{v}=(v,p)$, $\tilde{w}=(w,q)$,
$r=\tilde{d_c}(\tilde{v}, \tilde{w})$. Obviously,
$\tilde{w}\in\bar{B}^{\tilde{d_c}}(\tilde{v},r)$. Since
$w=\pi(\tilde{w})$, then  $w\in\bar{B}^{d_c}(v,r)$ due to
Proposition \ref{proj}, from where \eqref{ineq} follows. The
inequality \eqref{ineq_u} is proved in a similar way.
\end{proof}

The sketch of proof of the next theorem is similar to the proof of
its analog in \cite{bel}; the main difference lies in the method
of proof of the divergence of integral lines. In particular, we do
not need special polynomial ``privileged'' coordinates (though the
second-order coordinates, as well as coordinates constructed in
Proposition \ref{lift} are privileged as well) and do not use
Newton-type approximation methods.

\begin{Theorem}[\rm Local approximation theorem]\label{lat}
For the points $u,v,w\in U$, such that $d_c(u,v)=O(\varepsilon)$
and $d_c(u,w)=O(\varepsilon)$, the following estimate is true
$$|d_c(v,w)-d_c^u(v,w)|=O(\varepsilon^{1+\frac{1}{M}}).$$
\end{Theorem}
\begin{proof}

Let $\gamma:[0,1]\to\mathbb M$ be a geodesic for the distance
$d_c$, i. e.
$$\begin{cases}\dot{\gamma}(t)=\sum\limits_{i=1}^ma_i(t)X_i(\gamma(t)),\\
\gamma(0)=v,\ \gamma(1)=w
\end{cases}
$$
and $L(\gamma)=d_c(v,w)$. Consider a curve $\widehat{\gamma}(t)$
such that
$$\begin{cases}\dot{\widehat{\gamma}}(t)=\sum\limits_{i=1}^ma_i(t)
\widehat{X
}_i(\widehat{\gamma}(t)),\\
\widehat{\gamma}(0)=v
\end{cases}
$$ and denote $\widehat{w}=\widehat{\gamma}(1)$. Note that the lengths
of the curves $\gamma$ and $\widehat{\gamma}$ differ on a value of
order $O(\varepsilon^2)$ \cite{vk1}. Consequently,
$$d_c(v,w)= L(\gamma)= L(\widehat{\gamma})+O(\varepsilon^2)\geq
d_c^u(v,\widehat{w})\geq
d_c^u(v,w)-d_c^u(w,\widehat{w})+O(\varepsilon^2).$$ In a similar
way,
$$d_c^u(v,w)\geq d_c(v,w)-d_c(w,\widehat{w})+O(\varepsilon^2).$$
Taking in account Theorem \ref{int_lines_dc} and the estimates
\eqref{ineq} we get the required assertion.
\end{proof}

The following tangent cone result is proved in a similar way as
Theorem \ref{cc_teor} with the help of Theorem \ref{lat_dc} and
the homogeneity of the vector fields $\widehat{X}_i^u$.

\begin{Theorem}[\rm\cite{izv}]\label{cc_teor}
The metric space   $(U,d_c^g)$  is a local tangent cone at $g$ to
the metric space $({\cal U},d_c)$. The tangent cone has the
structure of a homogeneous space $G/H$, where $G$ is a  Carnot
group.

If $u$ is a regular point, the tangent cone is isomorphic to a
Carnot group.
\end{Theorem}

\end{document}